\documentclass[11pt,parskip=no]{scrartcl}
\usepackage[a4paper,hmargin={2.5cm,2.5cm},vmargin={2.5cm,2.5cm},heightrounded, marginparwidth=2.2cm, marginparsep=0.1cm]{geometry}

\usepackage[utf8]{inputenc}
\usepackage[T1]{fontenc}
\usepackage{lmodern}

\usepackage[square]{natbib} 
\usepackage[leqno]{amsmath}
\usepackage{amssymb,amsthm}

\usepackage{float} 

\usepackage{tikz-cd}
\usetikzlibrary{decorations.pathmorphing}

\RequirePackage{bbm}
\RequirePackage{mathtools}
\RequirePackage[mathscr]{euscript}
\RequirePackage{dsfont}
\RequirePackage{chngcntr}

\RequirePackage{amsthm}
\makeatletter%
\@ifclassloaded{beamer}%
{%
  \iflanguage{portuges}{%
    \theoremstyle{plain}
    \newtheorem{proposition}[theorem]{Proposição}
    \theoremstyle{definition}
    \newtheorem{assumption}[theorem]{Asserção}
    \theoremstyle{remark}
    \newtheorem{remark}[theorem]{Nota}
    \newtheorem*{remark*}{Nota}
    \newtheorem{remarks}[theorem]{Notas}}%
  {%
    \theoremstyle{plain}
    \newtheorem{proposition}[theorem]{Proposition}
    \theoremstyle{definition}
    \newtheorem{assumption}[theorem]{Assumption}
    \theoremstyle{remark}
    \newtheorem{remark}[theorem]{Remark}
    \newtheorem*{remark*}{Remark}
    \newtheorem{remarks}[theorem]{Remarks}}
  \RequirePackage{etoolbox}
  \AtBeginEnvironment{frame}{\setcounter{equation}{0}}
  
}%
{%
  \iflanguage{portuges}{%
    \theoremstyle{plain}
    \newtheorem{theorem}{Teorema}[section]
    \newtheorem{lemma}[theorem]{Lema}
    \newtheorem{proposition}[theorem]{Proposição}
    \newtheorem{corollary}[theorem]{Corolário}
    \theoremstyle{definition}
    \newtheorem{definition}[theorem]{Definição}
    \newtheorem{examples}[theorem]{Exemplos}
    \newtheorem{example}[theorem]{Exemplo}
    \theoremstyle{remark}
    \newtheorem{remark}[theorem]{Nota}}%
  {%
    \theoremstyle{plain}
    \newtheorem{theorem}{Theorem}[section]
    \newtheorem{lemma}[theorem]{Lemma}
    \newtheorem{proposition}[theorem]{Proposition}
    \newtheorem{corollary}[theorem]{Corollary}
    \theoremstyle{definition}
    \newtheorem{definition}[theorem]{Definition}
    \newtheorem{examples}[theorem]{Examples}
    \newtheorem{example}[theorem]{Example}
    \newtheorem{assumption}[theorem]{Assumption}
    \theoremstyle{remark}
    \newtheorem{remark}[theorem]{Remark}
    \newtheorem*{remark*}{Remark}
    }
  \RequirePackage[shortlabels]{enumitem}
  \counterwithin*{equation}{section}
  
  \newlist{tfae}{enumerate}{1}%
  \setlist[tfae,1]{label=(\roman*)}%
  \def\nlabel#1#2{\begingroup #2%
    \def\@currentlabel{#2}%
    \phantomsection\label{#1}\endgroup
  }%
  
  \newlist{conditions}{description}{1}%
  \setlist[conditions]{font=\normalfont\space,labelindent=\parindent}
}%
\makeatother%

\makeatletter

\def\slashedarrowfill@#1#2#3#4#5{%
  $\m@th\thickmuskip0mu\medmuskip\thickmuskip\thinmuskip\thickmuskip
   \relax#5#1\mkern-7mu%
   \cleaders\hbox{$#5\mkern-2mu#2\mkern-2mu$}\hfill
   \mathclap{#3}\mathclap{#2}%
   \cleaders\hbox{$#5\mkern-2mu#2\mkern-2mu$}\hfill
   \mkern-7mu#4$%
}

\newcommand*{\rightmodarrowfill@}{\slashedarrowfill@\relbar\relbar{\raisebox{0pc}{$\hspace{1pt}\circ$}}\rightarrow}
\newcommand*{\xmodto}[2][]{\ext@arrow 0055{\rightmodarrowfill@}{\;#1\;}{\;#2\;}}
\newcommand*{\lmodto}{\xmodto{\phantom{\to}}}
\newcommand*{\modto}{\xmodto{\;}}

\makeatother

\newcommand{\ff}{\mathfrak{f}}
\newcommand{\fv}{\mathfrak{v}}
\newcommand{\fx}{\mathfrak{x}}
\newcommand{\fy}{\mathfrak{y}}
\newcommand{\calL}{\mathcal{L}}
\newcommand{\calR}{\mathcal{R}}
\newcommand{\calS}{\mathcal{S}}
\DeclareMathOperator{\upc}{\uparrow\!}
\DeclareMathOperator{\downc}{\downarrow\!}
\DeclareMathOperator{\im}{im}
\DeclareMathOperator{\CoAlg}{CoAlg}
\newcommand*{\quot}[2]{{^{\textstyle #1}\big/_{\textstyle #2}}}
\DeclareMathOperator{\Mod}{Mod}
\newcommand{\catA}{\catfont{A}}
\newcommand{\catC}{\catfont{C}}
\newcommand{\catX}{\catfont{X}}
\newcommand{\SET}{\catfont{Set}}
\newcommand{\ORD}{\catfont{Ord}}
\newcommand{\POST}{\catfont{Pos}}
\newcommand{\MET}{\catfont{Met}}
\newcommand{\UMET}{\catfont{UMet}}
\newcommand{\BMET}{\catfont{BMet}}
\newcommand{\PROBMET}{\catfont{ProbMet}}
\newcommand{\TOP}{\catfont{Top}}
\newcommand{\STONE}{\catfont{BooSp}}
\newcommand{\SPEC}{\catfont{Spec}}
\newcommand{\PRIEST}{\catfont{Priest}}
\newcommand{\ESA}{\catfont{EsaSp}}
\newcommand{\GESA}{\catfont{GEsaSp}}
\newcommand{\COMPHAUS}{\catfont{CompHaus}}
\newcommand{\POSCH}{\catfont{PosComp}}
\newcommand{\COMPHAUSAB}{\catfont{CompHausAb}}
\newcommand{\STONEREL}{\catfont{BooSpRel}}
\newcommand{\PRIESTDIST}{\catfont{PriestDist}}
\newcommand{\ESADIST}{\catfont{EsaDist}}
\newcommand{\GESADIST}{\catfont{GEsaDist}}
\newcommand{\POSCHDIST}{\catfont{PosCompDist}}
\newcommand{\Cats}[1]{#1\text{-}\catfont{Cat}}
\newcommand{\Dists}[1]{#1\text{-}\catfont{Dist}}
\newcommand{\CatCHs}[1]{#1\text{-}\catfont{CatCH}}
\newcommand{\Priests}[1]{{#1\text{-}\catfont{Priest}}}
\newcommand{\FinSups}[1]{#1\text{-}\catfont{FinSup}}
\newcommand{\FinLats}[1]{#1\text{-}\catfont{FinLat}}
\newcommand{\BOOL}{\catfont{BA}}
\newcommand{\DLAT}{\catfont{DL}}
\newcommand{\DLO}{\catfont{DLO}}
\newcommand{\HEYT}{\catfont{HA}}
\newcommand{\CCD}{\catfont{CCD}}
\newcommand{\TAL}{\catfont{TAL}}
\newcommand{\FINSUP}{\catfont{FinSup}}
\newcommand{\CSALG}{C^*\text{-}\catfont{Alg}}
\newcommand{\op}{\mathrm{op}}
\newcommand{\sep}{\mathrm{sep}}
\newcommand{\comp}{\mathrm{comp}}
\newcommand{\fin}{\mathrm{fin}}
\newcommand{\cc}{\mathrm{cc}}
\newcommand{\dlat}{\mathrm{DL}}
\newcommand{\bool}{\mathrm{BA}}
\newcommand{\heyt}{\mathrm{HA}}
\newcommand{\two}{\catfont{2}}
\newcommand{\Pp}{\overleftarrow{[0,\infty]}_+}
\newcommand{\Pm}{\overleftarrow{[0,\infty]}_\wedge}
\newcommand{\V}{\mathcal{V}}
\newcommand{\quantale}{(\V,\otimes,k)}
\newcommand{\ddf}{\mathcal{D}}
\newcommand{\luk}{\odot}
\newcommand{\catfont}[1]{\mathsf{#1}}
\newcommand{\mate}[1]{\,^\ulcorner\! #1^\urcorner}
\newcommand{\doo}[1]{\overset{\centerdot}{#1}}
\newcommand{\ftC}{\functorfont{C}}
\newcommand{\ftD}{\functorfont{D}}
\newcommand{\ftH}{\functorfont{H}}
\newcommand{\ftU}{\functorfont{U}}
\newcommand{\ftV}{\functorfont{V}}
\newcommand{\ftII}[1]{{\catfont{|}#1\catfont{|}}}
\newcommand{\functorfont}{\mathsf}
\newcommand{\mH}{\monadfont{H}}
\newcommand{\mU}{\monadfont{U}}
\newcommand{\mV}{\monadfont{V}}
\newcommand{\monadfont}[1]{\mathbbm{#1}}
\newcommand{\umonad}{(\ftU,m,e)}
\DeclareMathAlphabet{\mathpzc}{OT1}{pzc}{m}{it}
\DeclareMathOperator{\coyoneda}{\mathpzc{h}}
\DeclareMathOperator{\coyonmult}{\mathpzc{w}}
\newcommand{\hmonad}{(\ftH,\coyonmult,\coyoneda)}
\newcommand{\vmonad}{(\ftV,\coyonmult,\coyoneda)}
\newcommand{\thU}{\theoryfont{U}}
\newcommand{\theoryfont}[1]{\mathscr{#1}}
\newcommand{\utheory}{(\mU,\V,\xi)}
\newcommand{\ZZ}{\field{Z}}
\newcommand{\RR}{\field{R}}
\newcommand{\TT}{\field{T}}
\newcommand{\field}[1]{\mathds{#1}}
\newcommand{\df}[1]{\emph{\textbf{#1}}}

\DeclareMathOperator{\Mnd}{Mnd}
\DeclareMathOperator{\LaxMnd}{LaxMnd}

\usepackage[backref=page,hypertexnames=false]{hyperref}

\hypersetup{
  colorlinks = true,
  citecolor= [rgb]{0,0.75,0}, 
  urlcolor=[rgb]{0.4,0,0.4}, 
  linkcolor=[rgb]{0.8,0,0} 
}

\begin{document}

\title{Duality theory for enriched Priestley spaces}

\author{Dirk Hofmann and Pedro Nora%
  \thanks{This work is supported by the ERDF -- European Regional Development
    Fund through the Operational Programme for Competitiveness and
    Internationalisation -- COMPETE 2020 Programme, by German Research Council
    (DFG) under project GO~2161/1-2, and by The Center for Research and
    Development in Mathematics and Applications (CIDMA) through the Portuguese
    Foundation for Science and Technology (FCT -- Fundação para a Ciência e a
    Tecnologia), references UIDB/04106/2020 and UIDP/04106/2020.}}

\publishers{\small{Center for Research and Development in Mathematics and
    Applications, Department of Mathematics, University of Aveiro, Portugal.

    Friedrich-Alexander-Universität Erlangen-Nürnberg, Germany.

    \href{mailto:dirk@ua.pt}{\texttt{dirk@ua.pt}}, \quad
    \href{mailto:pedro.nora@fau.de}{\texttt{pedro.nora@fau.de}}}}

\date{}

\maketitle


\begin{abstract}
  The term Stone-type duality often refers to a dual equivalence between a
  category of lattices or other partially ordered structures on one side and a
  category of topological structures on the other. This paper is part of a
  larger endeavour that aims to extend a web of Stone-type dualities from
  ordered to metric structures and, more generally, to quantale-enriched
  categories. In particular, we improve our previous work and show how certain
  duality results for categories of \([0,1]\)-enriched Priestley spaces and
  \([0,1]\)-enriched relations can be restricted to functions. In a broader
  context, we investigate the category of quantale-enriched Priestley spaces and
  continuous functors, with emphasis on those properties which identify the
  algebraic nature of the dual of this category.
\end{abstract}

\section{Introduction}
\label{sec:introduction}

Naturally, the starting point of our investigation of Stone-type dualities is
\citeauthor{Sto36}'s classical \citeyear{Sto36} duality result
\begin{equation}\label{d:eq:3}
  \STONE\sim\BOOL^{\op}
\end{equation}
for Boolean algebras and homomorphisms together with its generalisation
\begin{equation*}
  \SPEC\sim\DLAT^{\op}
\end{equation*}
to distributive lattices and homomorphisms obtained shortly afterwards in
\citep{Sto38}. Here \(\STONE\) denotes the category of Boolean
spaces\footnote{Also called Stone spaces in the literature, see \citep{Joh86},
  for instance.} and continuous maps, and \(\SPEC\) the category of spectral
spaces and spectral maps (see also \citep{Hoc69}). In this paper we will often
work with \emph{Priestley spaces} rather than with \emph{spectral spaces}, and
therefore consider the ``equivalent equivalence''
\begin{equation}\label{d:eq:7}
  \PRIEST\sim\DLAT^{\op}
\end{equation}
discovered in \citep{Pri70,Pri72}. There are many ways to deduce the duality
result \eqref{d:eq:3} from \eqref{d:eq:7}, we mention here one possibly lesser
known argument: in \citep{BGH92} it is observed that $\BOOL$ is the only
epi-mono-firm epireflective full subcategory of $\DLAT$, and, using that in both
$\STONE$ and $\PRIEST$ the epimorphisms are precisely the surjective morphisms,
an easy calculation shows that $\STONE$ is the only mono-epi-firm
mono-coreflective full subcategory of $\PRIEST$.

Exactly 20 years later, \citeauthor{Hal56} gave an extension of \eqref{d:eq:3}
to categories of \emph{continuous relations} between Boolean spaces and
\emph{hemimorphisms} between Boolean algebras, and a similar generalisation of
$\PRIEST\sim\DLAT^{\op}$ is described in \citep{CLP91}. Denoting by
\begin{itemize}
\item $\PRIESTDIST$ the category of Priestley spaces and continuous monotone
  relations, by
\item $\FINSUP$ the category of finitely cocomplete partially ordered sets and
  finite suprema preserving maps, and by
\item $\FINSUP_{\dlat}$ the full subcategory of $\FINSUP$ defined by all
  distributive lattices;
\end{itemize}
this result can be expressed as
\begin{equation}\label{d:eq:1}
  \PRIESTDIST\sim\FINSUP_{\dlat}^{\op}.
\end{equation}
We note that $\PRIESTDIST$ is precisely the Kleisli category $\PRIEST_{\mH}$ of
the Vietoris monad $\mH=\hmonad$ on $\PRIEST$, and that the functor
\(\PRIESTDIST \longrightarrow\FINSUP_{\dlat}^{\op}\) is a lifting of the
hom-functor \(\PRIESTDIST(-,1)\) into the one-element space. Furthermore, recall
that the two structures of a Priestley space -- the partial order and the
compact Hausdorff topology -- can be combined into a single topology: the
so-called downwards topology (see \citep{Jun04}, for instance). In particular,
the two-element Priestley space \(\two=\{0\le 1\}\) produces the Sierpiński
space \(\two\) with \(\{1\}\) closed, whereby the dual space \(\two^{\op}\) of
\(\two\) induces the topology on \(\{0,1\}\) with \(\{1\}\) being the only
non-trivial open subset. With this notation, the elements of the Vietories space
\(\ftH X\) of a Priestley space \(X\) can be identified with continuous maps
\(\varphi \colon X \longrightarrow\two\), whereby arrows of type \(X\lmodto 1\)
in \(\PRIESTDIST\) correspond to spectral maps
\(\psi \colon X \longrightarrow \ftH 1\simeq\two^{\op}\). In order to deduce the
equivalence \eqref{d:eq:1}, it is important to establish that there are
``enough'' spectral maps \(\psi \colon X \longrightarrow\two^{\op}\); in fact,
by definition, a partially ordered compact Hausdorff space \(X\) is Priestley
whenever the cone \((\psi \colon X \longrightarrow \two^{\op})_{\psi}\) is
point-separating and initial. Here it does not matter if we use \(\two\) or
\(\two^{\op}\) since \(\two\simeq\two^{\op}\) in \(\PRIEST\); however, when
moving to the quantale-enriched setting, the corresponding property does not
necessarily hold and therefore we must identify carefully if we refer to
\(\two\) or to \(\two^{\op}\).

Under the equivalence \eqref{d:eq:1}, continuous monotone \emph{functions}
correspond precisely to \emph{homomorphisms} of distributive lattices, therefore
the equivalence $\PRIEST\sim\DLAT^{\op}$ is a direct consequence of
\eqref{d:eq:1}. Furthermore, other well-known duality results can be obtained
from \eqref{d:eq:1} in a categorical way, we mention here the following
examples.
\begin{itemize}
\item As \eqref{d:eq:3} can be deduced from \(\PRIEST\sim\DLAT^{\op}\),
  \citeauthor{Hal56}'s duality
  \begin{displaymath}
    \STONEREL\sim\FINSUP_{\bool}^{\op}
  \end{displaymath}
  between the category \(\STONEREL\) of Boolean spaces and Boolean relations and
  the category \(\FINSUP_{\bool}\) of Boolean algebras and hemimorphisms (that
  is, the full subcategory of \(\FINSUP\) defined by all Boolean algebras) can
  be deduced from \eqref{d:eq:1}.
\item Combining $\PRIESTDIST\sim\FINSUP_{\dlat}^{\op}$ and
  $\PRIEST\sim\DLAT^{\op}$ gives immediately the duality result for distributive
  lattices with an operator (see \citep{Pet96,BKR07}).
\item The equivalence $\PRIESTDIST\sim\FINSUP_{\dlat}^{\op}$ has the
  surprising(?)\ consequence that $\PRIESTDIST$ is idempotent split
  complete. Hence, the idempotent split completion of $\STONEREL$ can be
  calculated as the full subcategory of $\PRIESTDIST$ defined by all split
  subobjects of Boolean spaces in $\PRIESTDIST$; likewise, the idempotent split
  completion of $\FINSUP_{\bool}$ can be taken as the full subcategory of
  $\FINSUP_{\dlat}$ defined by all split subobjects of Boolean algebras. Now, in
  the former case, these split subobjects are precisely the so-called Esakia
  spaces (see \citep{Esa74}), and in the latter case precisely the co-Heyting
  algebras (see \citep{MT46}); and we obtain a ``relational version'' of Esakia
  duality. We have described this more in detail in \citep{HN14}.
\end{itemize}
The situation is depicted in Figure~\ref{fig:Stone}.

\begin{figure}[H]
  \centering
  \begin{tikzcd}[cells={nodes={draw=black}}]
    \STONE\sim\BOOL^{\op} & \PRIEST\sim\DLAT^{\op} %
    \ar[ddl,squiggly,bend right] \ar[ddr,squiggly,bend left]
    \ar[l,squiggly]\\
    & \PRIESTDIST\sim\FINSUP_{\dlat}^{\op} %
    \ar[u,squiggly]%
    \ar[dl,squiggly] %
    \ar[d,squiggly]\\
    \CoAlg(\ftH)\sim\DLO^\op & \ESADIST\sim\FINSUP_\heyt^\op %
    \ar[r,squiggly] %
    & \ESA\sim\HEYT^\op
  \end{tikzcd}
  \caption{Stone type dualities}
  \label{fig:Stone}
\end{figure}
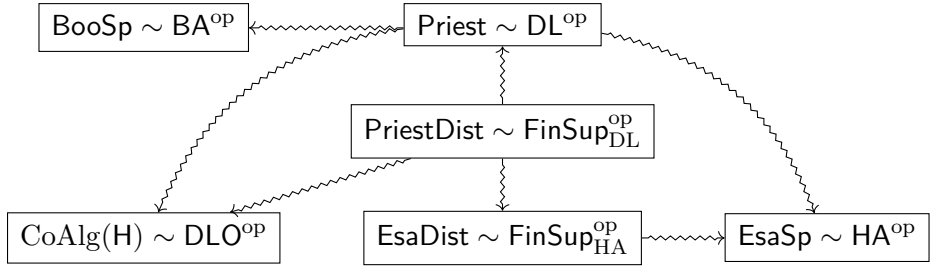

One might wish to consider all compact Hausdorff spaces in \eqref{d:eq:3}
instead of only the totally disconnected ones. Then the two-element space
\(\two=\{0,1\}\), respectively the two-element Boolean algebra, still induces
naturally an adjunction

\begin{equation*}
  \begin{tikzcd}[column sep=huge,ampersand replacement=\&]
    \COMPHAUS%
    \ar[r,"{\hom(-,\two)}", shift left, start anchor=east, end anchor=west, bend
    left=25, ""{name=U,below}]%
    \& \BOOL^{\op};%
    \ar[l,"{\hom(-,\two)}", start anchor=west,end anchor=east,shift left,bend
    left=25, ""{name=D,above}] \ar[from=U,to=D,"\bot" description]
  \end{tikzcd}
\end{equation*}
however, its restriction to the \emph{fixed} subcategories is precisely
\eqref{d:eq:3} (for the pertinent notions of duality theory we refer to
\citep{DT89, PT91}). In fact, by definition, a compact Hausdorff space $X$ is
Boolean whenever the cone $(f \colon X \longrightarrow \two)_{f}$ is
point-separating and initial with respect to the forgetful
$\COMPHAUS \longrightarrow\SET$; likewise, a partially ordered compact space $X$
is Priestley whenever the cone $(f \colon X \longrightarrow \two)_{f}$ is
point-separating and initial with respect to the forgetful
$\POSCH \longrightarrow\SET$ (equivalently, to the forgetful functor
$\POSCH \longrightarrow\COMPHAUS$).

In order to obtain a duality result for all compact Hausdorff spaces this way,
one needs to substitute the dualising object \(\two\) by a cogenerator in
$\COMPHAUS$, for instance, by the unit interval $[0,1]$ with the Euclidean
topology. Accordingly, one typically considers other types of algebras on the
dual side; \emph{i.e.} \(C^{*}\)-algebras instead of Boolean algebras. In
contrast, our aim is to develop a duality theory where one actually keeps the
``type of algebras'' in Figure~\ref{fig:Stone} but substitutes \emph{order} by
\emph{metric} everywhere; that is, one considers $[0,\infty]$-enriched
categories instead of $\two$-enriched categories (see \citep{Law73}). Therefore
one might attempt to create a network of dual equivalences

\begin{figure}[H]
  \centering
  \begin{tikzcd}[cells={nodes={draw=black}}]
    \COMPHAUS\sim (??)^{\op} & \POSCH\sim (??)^{\op} %
    \ar[ddl,squiggly,bend right] \ar[ddr,squiggly,bend left]
    \ar[l,squiggly]\\
    & \POSCHDIST\sim (??)^{\op} %
    \ar[u,squiggly]%
    \ar[dl,squiggly] %
    \ar[d,squiggly]\\
    \CoAlg(\ftH)\sim (??)^\op & \GESADIST\sim (??)^\op %
    \ar[r,squiggly] %
    & \GESA\sim (??)^\op
  \end{tikzcd}
  \caption{Metric Stone type dualities}
  \label{fig:EnrStone}
\end{figure}
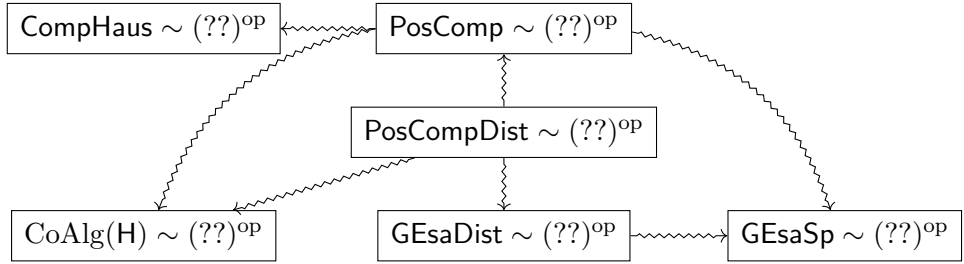

where each ``question mark category'' should be substituted by its metric
counterpart of Figure~\ref{fig:Stone}, or even better, a quantale-enriched
counterpart. For instance, for a quantale \(\V\), instead of \(\DLAT\) one would
expect a category of \(\V\)-categories with all ``finite'' weighted limits and
colimits and satisfying some sort of ``distributivity'' condition. Moreover,
these results should have the property that, when choosing the quantale
\(\V=\two\), we get the original picture of Figure~\ref{fig:Stone} back.

Unfortunately, the last requirement above does not make much sense \dots\ since
the picture of Figure~\ref{fig:EnrStone} is somehow inconsequential: both sides
of the equivalences should be generalised to corresponding metric or even
quantale-enriched versions. In particular, partially ordered compact spaces
should be substituted by their metric cousins as, for instance, studied in
\citep{HR18}. More specifically, we follow \citep{Tho09} and consider the
category \(\Cats{\V}^{\mU}\) of Eilenberg--Moore algebras and homomorphisms for
the ultrafilter monad $\mU$ on \(\Cats{\V}\). In analogy with the ordered case,
we call an $\mU$-algebra \emph{Priestley} whenever the cone of all homomorphisms
\(X \longrightarrow\V^{\op}\) in \(\Cats{\V}^{\mU}\) is point-separating and
initial. In \citep{HN18} we made an attempt to create at least parts of this
picture, for continuous quantale structures on the quantale \(\V=[0,1]\). In
Section~\ref{sec:dual-theory-enrich-conc} we improve slightly the results of
\citep{HN18} and show how certain duality results for categories of enriched
relations can be restricted to functions.

The classical duality results of \citeauthor{Sto36} and \citeauthor{Pri70} tell
us in particular that \(\STONE^\op\) and \(\PRIEST^\op\) are finitary
varieties. It is known since the late 1960's that also \(\COMPHAUS^\op\) is a
variety, not finitary but with rank \(\aleph_1\) (see \citep{Dus69,GU71});
however, this fact might not be obvious from the classical \citeauthor{Gel41a}
duality result
\begin{equation*}
  \COMPHAUS^{\op}\sim\CSALG
\end{equation*}
stating the equivalence between \(\COMPHAUS^{\op}\) and the category \(\CSALG\)
of commutative \(C^{\star}\)-algebras and homomorphisms. Nonetheless, it can be
deduced ``abstractly'' from the following well-known results.

\begin{theorem}
  \label{d:thm:6}
  A cocomplete category is equivalent to a quasivariety if and only if it has a
  regular projective regular generator.
\end{theorem}
\begin{proof}
  See, for instance, \citep[Theorem~3.6]{Ada04}.
\end{proof}

\begin{theorem}
  \label{d:thm:9}
  A category is a variety if and only if it is a quasivariety and has effective
  equivalence relations.
\end{theorem}
\begin{proof}
  See, for instance, \citep[Theorem~4.4.5]{Bor94a}
\end{proof}

Surprisingly, a similar investigation of \(\POSCH^{\op}\) was initiated only
recently: in \citep{HNN18} we show that \(\POSCH^{\op}\) is a \(\aleph_1\)-ary
quasivariety, and in \citep{Abb19, AR19_tmp} it is shown that \(\POSCH^{\op}\)
is indeed a \(\aleph_{1}\)-ary variety. In
Section~\ref{sec:dual-theory-enrich-abstr} we investigate the category
\(\Priests{\V}\) of \(\V\)-enriched Priestley spaces and morphisms, with
emphasis on those properties which identify \(\Priests{\V}^{\op}\) as some kind
of algebraic category.

\section{Quantale-enriched Priestley spaces}
\label{sec:quant-enrich-priestl}

In this section we recall the notions of quantale-enriched category and its
generalisation to compact Hausdorff spaces, which eventually leads to the notion
of \emph{quantale-enriched Priestley space} already studied in \citep{HN18,
  HN20}. We recall some of the basic definitions and properties, for more
information we refer to \citep{Kel82, Law73, Tho09} and \citep{HST14}.

\begin{definition}
  A \df{quantale} $\V=\quantale$ is a complete lattice $\V$ equipped with a
  commutative monoid structure $\otimes$, with identity $k$, so that, for each
  $u\in \V$,
  \begin{center}
    $u\otimes - \colon \V \longrightarrow \V$\quad has a right adjoint\quad
    $\hom(u,-) \colon \V \longrightarrow\V$.
  \end{center}
\end{definition}

\begin{definition}
  Let $\V=\quantale$ be a quantale.
  \begin{enumerate}
  \item A \df{$\V$-category} is a pair $(X,a)$ consisting of a set $X$ and a map
    $a \colon X\times X \longrightarrow\V$ satisfying
    \begin{align*}
      k\leq a(x,x) && \text{and}  && a(x,y)\otimes a(y,z)\leq a(x,z),
    \end{align*}
    for all \(x,y,z\in X\). Furthermore, a \(\V\)-category \((X,a)\) is called
    \df{separated} whenever
    \begin{displaymath}
      (k\le a(x,y)\quad \text{and}\quad k\le a(y,x))\implies x=y,
    \end{displaymath}
    for all \(x,y\in X\).
  \item A \df{$\V$-functor} $f \colon (X,a) \longrightarrow (Y,b)$ between
    $\V$-categories is a map $f \colon X \longrightarrow Y$ such that
    \[
      a(x,x')\leq b(f(x),f(x')),
    \]
    for all \(x,x'\in X\).
  \item Finally, $\V$-categories and $\V$-functors define the category
    $\Cats{\V}$, and its full subcategory defined by separated \(\V\)-categories
    is denoted by \(\Cats{\V}_{\sep}\).
  \end{enumerate}
\end{definition}

We note that there is a canonical forgetful functor
\(\Cats{\V} \longrightarrow\SET\) sending the \(\V\)-category \((X,a)\) to the
set \(X\). For every \(\V\)-category \(X=(X,a)\), the \df{dual} \(\V\)-category
\(X^{\op}\) is defined as \(X^{\op}=(X,a^{\circ})\) where
\begin{displaymath}
  a^{\circ}(x,y)=a(y,x),
\end{displaymath}
for all \(x,y\in X\). In fact, this construction defines a functor
\((-)^{\op}\colon\Cats{\V} \longrightarrow\Cats{\V}\) commuting with the
forgetful functor to \(\SET\).

\begin{examples}\label{d:exs:1}
  Below we list some of the principal examples, for more details we refer, for
  instance, to \citep{HR18}.
  \begin{enumerate}
  \item The two element chain $\two =\{0\le 1\}$ with $\otimes = \&$ and
    \(k=1\). Then $\Cats{\two}\sim\ORD$.
  \item The extended real half line $\overleftarrow{[0,\infty]}$ ordered by the
    ``greater or equal'' relation $\geqslant$ and
    \begin{itemize}
    \item the tensor product given by addition $+$, denoted by $\Pp$;
    \item or with $\otimes = \max$, denoted as $\Pm$.
    \end{itemize}
    Then $\Cats{\Pp}\sim\MET$ is the category of (generalised) metric spaces and
    non-expansive maps and $\Cats{\Pm}\sim\UMET$ is the category of
    (generalised) ultrametric spaces and non-expansive maps.
  \item The unit interval $[0,1]$ with the ``greater or equal'' relation
    $\geqslant$ and the tensor $u\oplus v=\min\{1,u+v\}$, denoted as
    $\overleftarrow{[0,1]}_\oplus$. Then
    $\Cats{\overleftarrow{[0,1]}_\oplus}\sim\BMET$ is the category of
    (generalised) bounded-by-one metric spaces and non-expansive maps.
  \item The unit interval $[0,1]$ with the usual order $\leqslant$ and
    $\otimes=\wedge$ the minimum, or $\otimes=*$ the usual multiplication, or
    $\otimes=\odot$ the Lukasiewicz sum defined by $u\odot v=\max\{0,u+v-1\}$.
    Then \(\Cats{[0,1]_\wedge} \sim \UMET\), \(\Cats{[0,1]_*} \sim \MET\), and
    \(\Cats{[0,1]_\odot} \sim \BMET\).
  \end{enumerate}
\end{examples}

\begin{example}\label{d:ex:2}
  The notion of probabilistic metric space goes back to \citep{Men42}. Here a
  \df{probabilistic metric} on a set \(X\) is a map
  $d \colon X\times X\times[0,\infty] \longrightarrow [0,1]$, where $d(x,y,t)=u$
  means that \emph{$u$ is the probability that the distance from $x$ to $y$ is
    less then $t$}. Similar to a classic metric, such a map is required to
  satisfy the following conditions:
  \begin{enumerate}
    \setcounter{enumi}{-1}
  \item $d(x,y,-) \colon [0,\infty] \longrightarrow [0,1]$ is left continuous,
  \item\label{d:prob:refl} $d(x,x,t)=1$ for $t>0$,
  \item\label{d:prob:trans} $d(x,y,r)*d(y,z,s)\le d(x,z,r+s)$,
  \item $d(x,y,t)=1=d(y,x,t)$ for all $t>0$ implies $x=y$,
  \item $d(x,y,t)=d(y,x,t)$ for all $t$,
  \item $d(x,y,\infty)=1$.
  \end{enumerate}
  The complete lattice
  \begin{displaymath}
    \ddf = \{f \colon [0,\infty] \longrightarrow [0,1] \mid
    f(t)=\bigvee_{s<t}f(s)\text{ for all }t\in[0,\infty]\}
  \end{displaymath}
  becomes a quantale with multiplication
  \begin{displaymath}
    (f\otimes g) (t)=\bigvee_{r+s\leqslant t}f(r)*g(s),
  \end{displaymath}
  for \(f,g\in\ddf\), and unit the map
  \(\kappa \colon [0,\infty] \longrightarrow [0,1]\) with \(\kappa(0)=0\) and
  \(\kappa(t)=1\) for $t>0$. In the formula above, one may substitute the
  multiplication \(*\) by any other tensor
  \(\otimes \colon [0,1]\times [0,1] \longrightarrow [0,1]\).

  Then a probabilistic metric can be seen as a map
  $d \colon X\times X \longrightarrow\ddf$, and conditions (\ref{d:prob:refl})
  and (\ref{d:prob:trans}) read as
  \begin{equation*}
    \kappa\le d(x,x)\qquad\text{and}\qquad d(x,y)\otimes d(y,z)\leq d(x,z).
  \end{equation*}
  Hence $\Cats{\ddf}\sim\PROBMET$ is the category of (generalised) probabilistic
  metric spaces and non-expansive maps.
\end{example}

Before adding a topological component to the theory of \(\V\)-categories, we
collect some well-known properties of \(\V\)-categories and \(\V\)-functors.

\begin{theorem}\label{d:thm:8}
  The canonical forgetful functor \(\Cats{\V}\longrightarrow\SET\) is
  topological. Here a cone
  \((f_{i} \colon (X,a) \longrightarrow (X_{i},a_{i}))_{i\in I}\) in
  \(\Cats{\V}\) is initial respect to \(\Cats{\V}\longrightarrow\SET\) if and
  only if, for all \(x,y\in X\),
  \begin{displaymath}
    a(x,y)=\bigwedge_{i\in I}a_{i}(f_{i}(x),f_{i}(y)).
  \end{displaymath}
  Therefore \(\Cats{\V}\) has concrete limits and colimits and a (surjective,
  initial monocone)-factorisation system; moreover,
  \(\Cats{\V}\longrightarrow\SET\) has a right adjoint
  \(\SET\longrightarrow\Cats{\V}\) (indiscrete structures) and a left adjoint
  \(\ftD \colon\SET\longrightarrow\Cats{\V}\) (discrete
  structures). Furthermore, a morphism \(f \colon X\longrightarrow Y\) in
  \(\Cats{\V}\) is
  \begin{enumerate}
  \item a monomorphism if and and only if \(f\) is injective,
  \item a regular monomorphism if and only if \(f\) is an embedding with respect
    to \(\Cats{\V}\longrightarrow\SET\),
  \item an epimorphism if and only if \(f\) is surjective.
  \end{enumerate}
\end{theorem}

\begin{proposition}\label{d:prop:1}
  The \(\V\)-category \(\V=(\V,\hom)\) is injective with respect to embeddings
  and, for every \(\V\)-category \(X\), the cone
  \((f \colon X \longrightarrow \V)_{f}\) is initial with respect to the
  forgetful functor \(\Cats{\V} \longrightarrow\SET\).
\end{proposition}

\begin{remark}
  Since \((-)^{\op}\colon\Cats{\V} \longrightarrow\Cats{\V}\) is a concrete
  isomorphism, Proposition~\ref{d:prop:1} applies also to the \(\V\)-category
  \(\V^{\op}\) \emph{in lieu} of \(\V\).
\end{remark}

In the remainder of this section we assume that \emph{the lattice \(\V\) is
  completely distributive}, we refer to \citep{Woo04} for the definition and an
extensive discussion of properties of this notion. In particular, under this
assumption it is useful to consider the \df{totally below} relation \(\lll\) on
the lattice \(\V\), which is defined by \( v \lll u \) whenever
\[
  u \leq \bigvee A \implies v \in \downc A,
\]
for every subset \(A\) of \(\V\).

\begin{assumption}\label{d:ass:1}
  The underlying lattice of the quantale \(\V\) is completely distributive.
\end{assumption}

\begin{remark}\label{d:rem:2}
  Regarding the various topologies on \(\V\) we have the following facts, for
  more information see \citep{GHK+03}.
  \begin{enumerate}
  \item \label{p:1} The Lawson topology on the completely distributive lattice
    \(\V\) is compact Hausdorff. With respect to this topology, as shown in
    \citep[Proposition~VII-3.10]{GHK+03}, an ultrafilter $\fv$ in \(\V\)
    converges to
    \begin{displaymath}
      \xi(\fv)=\bigwedge_{A\in\fv}\bigvee A \in \V.
    \end{displaymath}
    Moreover, the Scott topology respectively its dual topology have the
    following convergences:
    \begin{quote}
      \begin{tabbing}
        Scott topology: \hspace{5em} \= \(\fv \to x \iff \xi(\fv)\ge x\),\\
        Dual of Scott topology: \> \(\fv \to x \iff \xi(\fv)\le x\).\\
      \end{tabbing}
    \end{quote}
  \item By \citep[Lemma~VII-2.7]{GHK+03} and
    \citep[Proposition~VII-2.10]{GHK+03}, the Lawson topology of $\V$ coincides
    with the Lawson topology of $\V^\op$, and the set
    \begin{equation*}
      \{ \upc u \mid u \in \V \} \cup \{ \downc u \mid u \in \V \}
    \end{equation*}
    is a subbasis for the closed sets of this topology which is known as the
    interval topology.
  \item\label{d:item:1} The sets
    \begin{displaymath}
      \upc v=\{u\in\V\mid v\le u\}\qquad (v\in\V)
    \end{displaymath}
    form a subbase for the closed sets of the dual of the Scott topology of $\V$
    (see \citep[Proposition~VI-6.24]{GHK+03}). We denote (the convergence of)
    this topology by \(\xi_{\le}\).
  \item The convergence \(\xi \colon \ftU\V \longrightarrow\V\) together with
    the ultrafilter monad \(\mU=\umonad\) and the quantale \(\V\) defines a
    topological theory in the sense of \citep{Hof07}, and therefore allows for
    an extension of the ultrafilter monad $\mU$ to a monad on \(\Cats{\V}\) (see
    \citep{Tho09}). We denote the corresponding Eilenberg--Moore category
    \(\Cats{\V}^{\mU}\) by $\CatCHs{\V}$, and refer to its objects as
    \df{\(\V\)-categorical compact Hausdorff spaces} (see also \citep{HR18}). In
    more detail, a $\V$-categorical compact Hausdorff space is a triple
    $(X,a,\alpha)$ where
    \begin{itemize}
    \item $(X,a)$ is a $\V$-category and
    \item $\alpha \colon UX \longrightarrow X$ is the convergence of a compact
      Hausdorff topology on $X$ such that
      $\alpha \colon (UX,Ua) \longrightarrow (X,a)$ is a $\V$-functor.
    \end{itemize}
  \end{enumerate}
\end{remark}

\begin{example}
  The triple \(\V=(\V,\hom,\xi)\) is a $\V$-categorical compact Hausdorff
  space. Moreover, for a \(\V\)-categorical compact Hausdorff space
  \(X=(X,a,\alpha)\), also \(X^{\op}=(X,a^{\circ},\alpha)\) is a
  \(\V\)-categorical compact Hausdorff space.
\end{example}

\begin{example}
  As it is pointed out in \citep{Tho09}, \(\two\)-categorical compact Hausdorff
  spaces are precisely \citeauthor{Nac65}'s ordered compact Hausdorff spaces.
\end{example}

\begin{proposition}\label{p:0}
  For a quantale \(\V\), the sets
  \begin{equation*}
    \{ u \in \V \mid v \lll u \} \qquad  (v \in \V)
  \end{equation*}
  form a subbase for its Scott topology.
\end{proposition}

\begin{proof}
  We start by proving that for every \(v \in \V\) the set
  \(\{ u \in \V \mid v \lll u \}\) is open. Let \(\fv\) be an ultrafilter in
  \(\V\) that converges to \(u\in\V\) such that \(v \lll u \).  The properties
  of the totally below relation guarantee that there exists \(w \in \V\) such
  that \(v \lll w \lll u \). Then, by Remark~\ref{d:rem:2}~(\ref{p:1}), for
  every \(A \in \fv\), \(u \leq \bigvee A \). Hence, for every \(A \in \fv \)
  there exists \(a \in A\) such that \(w \leq a\).  Therefore, for every
  \(A \in \fv\),
  \[
    A \cap \{ u \in \V \mid v \lll u \} \neq \varnothing.
  \]

  We show now that the sets \(\{ u \in \V \mid v \lll u \}\) \((v\in\V)\) induce
  the convergence of the Scott topology. Let \(w\) be an element of \(\V\) and
  \(\fv\) and ultrafilter on \(\V\) such that, for every \(v \lll w\) in \(\V\),
  the set \(\{ u \in \V \mid v \lll u \}\) belongs to \(\fv\). Then, since
  \(\V\) is completely distributive, we have
  \[
    w = \bigvee_{v \lll w} v \leq \bigvee_{A \in \fv} \bigwedge A =
    \xi(\fv).\qedhere
  \]
\end{proof}

\begin{remark}
  For a point-separating cone
  \((f_{i} \colon (X,a,\alpha) \longrightarrow (X_{i},a_{i},\alpha_{i}))_{i\in
    I}\) in \(\CatCHs{\V}\), the following assertions are equivalent, for
  details see \citep{Tho09}.
  \begin{tfae}
  \item For all \(x,y\in X\),
    \(\displaystyle{a(x,y)=\bigwedge_{i\in I}a_{i}(f_{i}(x),f_{i}(y))}\).
  \item
    \((f \colon (X,a,\alpha) \longrightarrow (X_{i},a_{i},\alpha_{i}))_{i\in
      I}\) is initial with respect to \(\CatCHs{\V}\longrightarrow\COMPHAUS\).
  \item
    \((f \colon (X,a,\alpha) \longrightarrow (X_{i},a_{i},\alpha_{i}))_{i\in
      I}\) is initial with respect to \(\CatCHs{\V}\longrightarrow\SET\).
  \end{tfae}
  In the sequel we will simply say ``initial'' when referring to either of these
  forgetful functors. We also note that a cone
  \((f_{i} \colon (X,a,\alpha) \longrightarrow (X_{i},a_{i},\alpha_{i}))_{i\in
    I}\) is point-separating if and only if it is a monocone in \(\CatCHs{\V}\).
\end{remark}

\begin{theorem}
  The category \(\CatCHs{\V}\) is monadic over \(\Cats{\V}\) and topological
  over \(\COMPHAUS\), hence \(\CatCHs{\V}\) is complete and cocomplete and has a
  (surjective, initial monocone)-factorisation system.
\end{theorem}
\begin{proof}
  See \citep{Tho09}.
\end{proof}

\begin{definition}
  A \(\V\)-categorical compact Hausdorff space $X$ is called \df{Priestley}
  whenever the cone $\CatCHs{\V}(X,\V^\op)$ is point-separating and initial with
  respect to $\CatCHs{\V} \longrightarrow\COMPHAUS$.
\end{definition}

\begin{example}
  For \(\V=\two\), the notion of Priestley space coincides with the classical
  one.
\end{example}

\begin{remark}
  By definition, the \(\V\)-categorical compact Hausdorff space \(\V^{\op}\) is
  Priestley. Moreover, every finite separated \(\V\)-categorical compact
  Hausdorff space is Priestley.
\end{remark}

We denote the full subcategory of $\CatCHs{\V}$ defined by all Priestley spaces
by $\Priests{\V}$. Due to well-known facts about factorisation structures for
cones (see \citep{AHS90}), we have the following:

\begin{proposition}
  The full subcategory \( \Priests{\V}\) of \(\CatCHs{\V} \) is reflective.
\end{proposition}
We denote the left adjoint of the inclusion functor
\(\Priests{\V} \longrightarrow\CatCHs{\V}\) by
$\pi_0 \colon \CatCHs{\V}\longrightarrow\Priests{\V}$.
\begin{proof}
  For each $X$ in $\CatCHs{\V}$, its reflection $X \longrightarrow\pi_{0}(X)$
  into $\Priests{\V}$ is given by the (surjective, initial
  monocone)-factorisation of the cone
  $(\varphi \colon X \longrightarrow \V^{\op})_{\varphi}$ of all morphisms from
  $X$ to $\V^{\op}$ in $\CatCHs{\V}$.
  \begin{equation*}
    \begin{tikzcd}
      X \ar[r,->>] \ar[rr,"\varphi",bend left] %
      & \pi_0(X) \ar[r,"\widetilde{\varphi}"'] %
      & {\V^{\op}}
    \end{tikzcd}
  \end{equation*}
  To show that this construction defines indeed a left adjoint to
  $\Priests{\V} \longrightarrow\CatCHs{\V}$, consider
  $f \colon X \longrightarrow Y$ in $\CatCHs{\V}$ where $Y$ is Priestley. Then,
  for every $\varphi \colon Y \longrightarrow \V^{\op}$, there is some arrow
  $\widetilde{\varphi}\colon\pi_{0}(X) \longrightarrow \V^{\op}$ making the
  diagram
  \begin{equation}
    \label{d:eq:4}
    \begin{tikzcd}
      X \ar[r,->>]\ar[d,"f"'] & \pi_{0}(X)
      \ar[d,"\widetilde{\varphi}"] \\
      Y \ar[r,"\varphi"'] & {\V}^{\op}
    \end{tikzcd}
  \end{equation}
  commute. Since the top arrow of \eqref{d:eq:4} is surjective and the cone
  $(\varphi \colon Y \longrightarrow \V^{\op})_{\varphi}$ is point-separating
  and initial, there is a diagonal arrow
  $\bar{f}\colon\pi_{0}(X) \longrightarrow Y$ in \eqref{d:eq:4} making in
  particular the diagram
  \begin{equation*}
    \begin{tikzcd}
      X \ar[r,->>]\ar[d,"f"'] & \pi_{0}(X)\ar[dl,"\bar{f}"]\\
      Y
    \end{tikzcd}
  \end{equation*}
  commute.
\end{proof}

\begin{corollary}\label{d:cor:1}
  The category \(\Priests{\V}\) is complete and cocomplete.
\end{corollary}

We already observed in \citep[Remark~4.52]{HN20} that a monocone in
\(\Priests{\V}\) is initial with respect to \(\Priests{\V} \longrightarrow\SET\)
if and only if it is initial with respect to $\CatCHs{\V} \longrightarrow\SET$
(the same argument as in the proof of \citep[Theorem~A.6]{HN20} applies
here). At this moment we do not know whether, for instance, every separated
metric compact Hausdorff space is Priestley. However, since $[0,1]^\op$ is an
initial cogenerator in $\POSCH$ (see \citep{Nac65}), we have the following fact.

\begin{proposition}
  The inclusion functor $\POSCH \longrightarrow\CatCHs{[0,1]}$ corestricts to
  $\POSCH \longrightarrow\Priests{[0,1]}$.
\end{proposition}

\section{Duality theory for enriched Priestley spaces: concretely}
\label{sec:dual-theory-enrich-conc}

In this section we build on the duality results of \citep{HN18} for Priestley
spaces enriched in the complete lattice $[0,1]$ with a continuous quantale
structure $\otimes \colon [0,1]\times [0,1] \longrightarrow [0,1]$ and with
neutral element $1$. We recall some of the principal results, and then, for the
\L{}ukasiewicz tensor on \([0,1]\), show how to restrict the \(\V\)-relational
duality results obtained in \citep[Section~9]{HN18} to categories of functions.

In analogy with the classical situation, our starting point is the category
$\FinSups{[0,1]}$ of finitely cocomplete $[0,1]$-categories and $[0,1]$-functors
that preserve finite weighted colimits.

\begin{theorem}
  The category $\FinSups{[0,1]}$ is a $\aleph_1$-ary quasivariety.
\end{theorem}
\begin{proof}
  See \citep[Remark~2.10]{HN18}.
\end{proof}

In the sequel we consider the Vietoris monad \(\mH=\hmonad\) on the category
\(\POSCH\) of partially ordered compact Hausdorff spaces and monotone continuous
maps, more information on power constructions in topology can be found in
\citep{Sch93, Sch93a}. In our previous work \citep{HN18, HNN19} we used the
notation \(\mV\) instead of \(\mH\); however, in this paper we think of the
classic Vietoris topology \citep{Vie22} as an extension of the Hausdorff metric
and reserve the designation \(\mV\) for the monad based on presheafs
\(X \longrightarrow\V\) rather than subsets \(A\subseteq X\). Similarly to the
$\two$-enriched case mentioned in the Introduction, we obtain the commutative
diagram
\begin{displaymath}
  \begin{tikzcd}[column sep=large]
    \POSCH_\mH \ar[r,"\ftC"]
    & \FinSups{[0,1]}^\op \\
    \POSCH \ar[u] \ar[ur,"{\ftC=\hom(-,[0,1]^\op)}"']
  \end{tikzcd}
\end{displaymath}
of functors. However, unlike
\(\hom(-,1)\colon\PRIEST_{\mH}\longrightarrow\FINSUP^{\op}\), the functor
$\ftC \colon\POSCH_\mH \longrightarrow \FinSups{[0,1]}^\op$ is not fully
faithful, as the next example shows.

\begin{example}\label{d:ex:1}
  As observed in \citep[Example~6.16]{HN18}, for every $u\in [0,1]$, the map
  $u\otimes- \colon [0,1] \longrightarrow [0,1]$ is a morphism in
  $\FinSups{[0,1]}$ sending $1$ to $u$. On the other hand, there are only two
  morphisms of type $1\lmodto 1$ in $\POSCH_\mH$.
\end{example}

Therefore we have to consider further structure on the right-hand side. The
starting point is the following observation.

\begin{theorem}
  The category $\FinSups{[0,1]}$ has a bimorphism representing monoidal
  structure.
\end{theorem}
\begin{proof}
  See \citep[Section~6.5]{Kel82}.
\end{proof}

This leads us to the category
\begin{equation*}
  \Mnd(\FinSups{[0,1]})
\end{equation*}
of monoids and homomorphisms in $\FinSups{[0,1]}$ with respect to the
above-mentioned monoidal structure and \emph{with neutral element the
  top-element}, and to the category
\begin{equation*}
  \LaxMnd(\FinSups{[0,1]})
\end{equation*}
with the same objects as $\Mnd(\FinSups{[0,1]})$, but now with morphisms those
of $\FinSups{[0,1]}$ that preserve the monoid structure laxly:
\begin{equation*}
  \Phi(\psi_1\otimes \psi_2)\le\Phi(\psi_1)\otimes\Phi(\psi_2).
\end{equation*}
We obtain the commutative diagram
\begin{displaymath}
  \begin{tikzcd}[column sep=large, row sep=large]
    \POSCH_\mH \ar[r,"\ftC"]\ar[dotted, bend right=50]{d}
    & \LaxMnd(\FinSups{[0,1]})^\op \ar[dotted]{dl}[rotate=30, yshift=-2ex]{\top}\\
    \POSCH \ar{u}[xshift=-0.7ex]{\dashv} \ar[bend right]{ur}[swap]{\ftC=\hom(-,[0,1]^\op)}
  \end{tikzcd}
\end{displaymath}
of functors (represented by solid arrows), and the induced monad morphism
$j=(j_{X})_{X}$ is given by the family of maps
\begin{center}
  $j_X \colon \ftH X\longrightarrow[\ftC X,[0,1]],\;A\longmapsto\Phi_A$,
\end{center}
with
$\Phi_A \colon \ftC X \longrightarrow[0,1],\;\psi\longmapsto\sup_{x\in
  A}\psi(x)$.

\begin{proposition}
  Let $X$ be in $\POSCH$ and $A\subseteq X$ closed and upper. Then $A$ is
  irreducible if and only if $\Phi_A$ satisfies
  \begin{center}
    $\Phi_A(1)=1$\quad and \quad
    $\Phi_A(\psi_1\otimes \psi_2)=\Phi_A(\psi_1)\otimes\Phi_A(\psi_2)$.
  \end{center}
\end{proposition}
\begin{proof}
  See \citep[Proposition~6.7]{HN18}.
\end{proof}

\begin{corollary}
  Let $\varphi \colon X\lmodto Y$ be a morphism in $\POSCH_{\mH}$. Then $\varphi$
  is a function if and only if $C\varphi$ is a morphism in
  $\Mnd(\FinSups{[0,1]})$.
\end{corollary}

\begin{theorem}\label{d:thm:7}
  For $\otimes=*$ or $\otimes=\luk$, the monad morphism $j$ is an
  isomorphism. Therefore the functors
  \begin{equation*}
    \ftC \colon \POSCH_{\mH} \longrightarrow\LaxMnd(\FinSups{[0,1]})^\op
    \quad\text{and}\quad
    \ftC \colon \POSCH \longrightarrow\Mnd(\FinSups{[0,1]})^\op
  \end{equation*}
  are fully faithful.
\end{theorem}
\begin{proof}
  See \citep[Theorem~6.14 and Corollary~6.15]{HN18}.
\end{proof}

Theorem~\ref{d:thm:7} does not extend to arbitrary continuous quantale
structures on $[0,1]$ since, by Example~\ref{d:ex:1}, the functor
$\ftC\colon\POSCH_\mH \longrightarrow \LaxMnd(\FinSups{[0,1]_\wedge})^\op$ is
not full. In fact, this example also shows that its restriction
$\ftC\colon\COMPHAUS_\mH \longrightarrow \LaxMnd(\FinSups{[0,1]_\wedge})^\op$ to
compact Hausdorff spaces is not full. However, passing from relations to
functions improves the situation: it is shown in \citep{Ban83} that the functor
$\ftC \colon\COMPHAUS \longrightarrow \Mnd(\FinSups{[0,1]_\wedge})^\op$ is fully
faithful (see also \citep[Remark~2.7]{HN18}). This result generalises to our
setting.

\begin{theorem}
  \label{p:2}
  The functor $\ftC \colon\COMPHAUS \longrightarrow\Mnd(\FinSups{[0,1]})^\op$ is
  fully faithful.
\end{theorem}
\begin{proof}
  See \citep[Theorem~6.23]{HN18}.
\end{proof}

\begin{remark}
  To identify the image of the functor of Theorem~\ref{p:2}, we can proceed as
  in Section~7 of \citep{HN18}, although with a small adjustment.  Since we
  consider now ``initial with respect to \(\COMPHAUS \longrightarrow \SET \)''
  instead of ``initial with respect to \(\POSCH \longrightarrow \SET \)'' in
  \citep[Lemma~7.3 and 7.4]{HN18}, at the beginning of the proof we do not kown
  whether ``\(\psi(x) > \psi(y)\) or \(\psi(x) < \psi(y)\)''. We can remedy the
  situation by requiring that \(L\) is also closed in \(CX\) under an additional
  unary operation, which in \([0,1]\) is interpreted as \(u \mapsto 1 - u \).
  This new operation acts as a ``complement'', and introducing it corresponds to
  the passage from distributive lattices to Boolean algebras in the classical
  case (see also~\citep[Example~3.5]{Hof02}).
\end{remark}

However, \citeauthor{Ban83}'s result does not extend to partially ordered
compact spaces, as the following example shows.

\begin{example}
  The functor
  $\ftC\colon\POSCH \longrightarrow \Mnd(\FinSups{[0,1]_\wedge})^\op$ is not
  full. As pointed out in \citep[Example~6.16]{HN18}, for the separated ordered
  compact space $X=\{0 > 1\}$,
  \begin{equation*}
    \ftC X= \{(u,v)\in [0,1]\times[0,1] \mid u \leq v\}.
  \end{equation*}
  $\ftH X$ contains three elements; however, for every $w\in [0,1]$, the map
  \begin{equation*}
    \Phi_w \colon\ftC X\longrightarrow[0,1],\;(u,v)\longmapsto
    u\vee\left(w\wedge v\right)
  \end{equation*}
  is a morphism in $\Mnd(\FinSups{[0,1]_\wedge})$ with $\Phi_w(0,1)=w$.
\end{example}

\begin{theorem}
  We consider an additional operation $\ominus$ in our theory (which is
  interpreted as truncated minus in $[0,1]$). Then
  $\ftC\colon\POSCH_\mH \longrightarrow \LaxMnd_\ominus(\FinSups{[0,1]})^\op$ is
  fully faithful.
\end{theorem}
\begin{proof}
  See \citep[Theorem~6.19][]{HN18}.
\end{proof}

As we already observed in \citep{HN18}, the setting above is not really
consequential since we still consider \emph{ordered} compact Hausdorff spaces as
well as the Vietoris functor based on \emph{subsets}, that is, continuous
functions into the Sierpiński space \(\two\). We obtain results closer to the
classical case by also enriching the topological side. That is, we consider
enriched Priestley spaces and the \emph{enriched} Vietoris monad
$\mV=\vmonad$. The latter is introduced in \citep{Hof14} in the context of
\(\thU\)-categories and \(\thU\)-functors, for a \emph{topological theory}
\(\thU=\utheory\) based on the ultrafilter monad \(\mU=\umonad\). For an
overview of the background theory we refer to \citep[Section~1]{Hof14}, and
mention here only that
\begin{itemize}
\item an \(\thU\)-category \((X,a)\) is given by a set \(X\) and a map
  \(a \colon \ftU X\times X \longrightarrow \V\) satisfying two axioms similar
  to the ones of a \(\V\)-category,
\item the category of \(\thU\)-categories and \(\thU\)-functors is denoted as
  \(\Cats{\thU}\),
\item by combining the internal hom and the convergence
  \(\xi \colon \ftU\V\longrightarrow\V\), the quantale \(\V\) becomes an
  \(\thU\)-category where \((\fv,v)\longmapsto\hom(\xi(\fv),v)\),
\item the underlying set of $\ftV X$ is the set
  \begin{equation*}
    \{\text{all $\thU$-functors $\varphi \colon X \longrightarrow \V$}\}.
  \end{equation*}
\end{itemize}
For \(\V=\two\), \(\thU\)-categories correspond to topological spaces and
\(\thU\)-functors to continuous maps (see \citep{Bar70}), and \(\V=\two\) is the
Sierpiński space \(\two\) where \(\{1\}\) is closed. On the other hand, for the
multiplication \(*\) on \([0,1]\), an \(\thU\)-category is essentially an
approach space (see \citep{Low97}), thanks to the isomorphism of quantales
\([0,1]_{*}\simeq\Pp\).

\begin{remark}
  An interesting connection between topological theories and lax distributive
  laws is exposed in \citep{Tho19}.
\end{remark}

\begin{theorem}
  If $\otimes=*$, \(\otimes=\wedge\) or $\otimes=\luk$, then the monad
  $\mV=\vmonad$ on \(\Cats{\thU}\) restricts to $\Priests{[0,1]}$.
\end{theorem}
\begin{proof}
  See \citep[Corollary~9.7]{HN18}.
\end{proof}

We obtain now the commutative diagram
\begin{displaymath}
  \begin{tikzcd}[column sep=large, row sep=large]
    \Priests{[0,1]}_\mV \ar[r,"\ftC"]\ar[dotted, bend right=50]{d}
    & \FinSups{[0,1]}^\op \ar[dotted]{dl}[rotate=30, yshift=-2ex]{\top} \\
    \Priests{[0,1]} \ar{u}[xshift=-0.7ex]{\dashv} \ar[bend right]{ur}[swap]{\ftC=\hom(-,[0,1]^\op)}
  \end{tikzcd}
\end{displaymath}
of functors (represented by solid arrows), we stress that here the functor
\(\ftC \colon \Priests{[0,1]}_\mV \longrightarrow\FinSups{[0,1]}^\op\) is a
lifting of the hom-functor \(\hom(-,1)\). The $X$-component of the induced monad
morphism $j$ is given by
\begin{equation*}
  j_X \colon \ftV X\longrightarrow [\ftC X,[0,1]],\quad
  (\varphi \colon
  1\modto X)\longmapsto\left(\psi\mapsto\psi\cdot\varphi=\bigvee_{x\in
      X}(\psi(x)\otimes\varphi(x))\right).
\end{equation*}

\begin{theorem}\label{d:thm:3}
  If $\otimes=*$ or $\otimes=\luk$, then the monad morphism $j$ is an
  isomorphism. Consequently, the functor
  \begin{equation}\label{d:eq:6}
    \ftC \colon \Priests{[0,1]}_\mV \longrightarrow\FinSups{[0,1]}^\op
  \end{equation}
  is fully faithful.
\end{theorem}
\begin{proof}
  See \citep[Theorem~9.10]{HN18}.
\end{proof}

We recall that, in the classic case
\(\PRIESTDIST \longrightarrow \FINSUP^{\op}\) mentioned in the Introduction, we
can first restrict the objects on the right-hand side to (distributive)
lattices, and then observe that those continuous distributors coming from
continuous monotone maps correspond precisely to lattice homomorphisms on the
right-hand side. We aim now at a similar result for the fully faithful functor
\eqref{d:eq:6}. To do so, we wish to identify those $[0,1]$-functors
$\Phi \colon \ftC X \longrightarrow [0,1]$ which correspond to ``the points of
$X$ inside $\ftV X$''; that is, to the $\thU$-functors of the form
$a(\doo{x},-) \colon X \longrightarrow [0,1]$. For now we are only able to do so
for the \L{}ukasiewicz tensor on \([0,1]\) employing the fact that the quantale
$[0,1]_\luk$ is a \df{Girard quantale}: for every $u\in [0,1]$,
$u=\hom(\hom(u,\bot),\bot)$. We recall that $\hom(u,\bot)=1-u$ and put
$u^\bot=1-u$. Also note that $(-)^\bot \colon [0,1] \longrightarrow [0,1]^\op$
is an isomorphism in $\Priests{[0,1]_\luk}$.

In a nutshell, our strategy is the same as in the ordered case: we show that an
additional property on $\Phi$ translates to
``$\varphi \colon X \longrightarrow [0,1]$ is irreducible'', and ``soberness''
of $X$ guarantees $\varphi=a(\doo{x},-)$, for some $x\in X$. Hence, we need to
introduce these notions for $\thU$-categories, which fortunately was already
done in \citep{CH09}. In our context, ``sober'' means \emph{Cauchy-complete}
(called \emph{Lawvere complete} in \citep{CH09}) and ``irreducible'' means
\emph{left adjoint $\thU$-distributor}. We do not introduce these notions here
but rather refer to the before-mentioned literature; however, we recall the
following two results.

\begin{theorem}\label{d:thm:1}
  An $\thU$-functor $\varphi \colon X \longrightarrow [0,1]$ (viewed as an
  $\thU$-distributor from $1$ to $X$) is left adjoint if and only if the
  representable $[0,1]$-functor
  \begin{equation*}
    [\varphi,-] \colon \Cats{\thU}(X,[0,1]) \longrightarrow [0,1],\quad
    \varphi' \longmapsto \bigwedge_{x\in X}\hom(\varphi(x),\varphi'(x))
  \end{equation*}
  preserves copowers and finite suprema.
\end{theorem}
\begin{proof}
  See \citep[Proposition~3.5]{HS11}.
\end{proof}

\begin{theorem}
  Every \(\V\)-categorical compact Hausdorff space $X$ is Cauchy complete
  (viewed as an $\thU$-category); that is, every left adjoint $\thU$-distributor
  $\varphi$ from $1$ to $X$ is of the form $\varphi=a(\doo{x},-)$, for some
  $x\in X$.
\end{theorem}
\begin{proof}
  See \citep[Corollary~4.18]{HR18}.
\end{proof}

Let now $\varphi \colon X \longrightarrow [0,1]$ be an $\thU_\luk$-functor. To
link Theorem~\ref{d:thm:1} with our situation, we view $\varphi$ as a
$[0,1]_\luk$-distributor $\varphi \colon 1\lmodto X$ and note that
\begin{equation*}
  \begin{tikzcd}[row sep=large, column sep=large,ampersand replacement=\&]
    \Dists{[0,1]}(X,1) \ar[r,"{(-)^\bot}"]\ar[d,"(-\cdot\varphi)"'] \&
    \Dists{[0,1]}(1,X)^\op
    \ar[d,"[{\varphi,-]^\op}"] \\
    {[0,1]} \ar[r,"{(-)^\bot}"'] \& {[0,1]}^\op
  \end{tikzcd}
\end{equation*}
commutes in $\Cats{[0,1]_\luk}$ (see
\citep[Proposition~4.35]{HR18}). Furthermore, we can restrict the top line of
diagram above to the $[0,1]_\luk$-functor
\begin{equation*}
  (-)^\bot \colon \Cats{\thU_\luk}(X,[0,1]^\op) \longrightarrow
  \Cats{\thU_\luk}(X,[0,1])^\op,
\end{equation*}
which implies at once:

\begin{proposition}
  An $\thU_\luk$-functor $\varphi \colon X \longrightarrow [0,1]$ is a left
  adjoint $\thU_\luk$-distributor $\varphi \colon 1\lmodto X$ if and only if the
  $[0,1]_\luk$-functor
  $(-\cdot\varphi)\colon\Cats{\thU}(X,[0,1]^\op) \longrightarrow [0,1]$
  preserves powers and finite infima.
\end{proposition}

Finally, for an object $X$ in $\Priests{[0,1]_\luk}$, we will show that the
inclusion $[0,1]_\luk$-functor
$\ftC X\hookrightarrow\Cats{\thU_\luk}(X,[0,1]^\op)$ is $\bigvee$-dense. This
property guarantees that
$-\cdot\varphi \colon\Cats{\thU_\luk}(X,[0,1]^\op) \longrightarrow [0,1]$
preserves powers and finite infima if and only if
$-\cdot\varphi \colon CX \longrightarrow [0,1]$ does so.

For every $\thU_\luk$-category $(X,a)$, the $[0,1]_\luk$-subcategory
\begin{equation}\label{d:eq:2}
  \{\text{all $\thU_\luk$-functors $\varphi \colon X \longrightarrow [0,1]$}\}
  \subseteq [0,1]^X
\end{equation}
is closed under weighted limits and finite weighted colimits; we shall show now
that this property characterises the collection of all $\thU_\luk$-functors
$\varphi \colon X \longrightarrow [0,1]$. This way we transport a well-known
fact from approach theory to the ``\L{}ukasiewicz setting'' (see \citep{Low97}).

In general, every $[0,1]$-subcategory $\calR\subseteq [0,1]^X$ closed under
weighted limits and finite weighted colimits corresponds to a monad
\begin{equation*}
  \mu \colon [0,1]^X \longrightarrow [0,1]^X\qquad (1\le\mu,\;\mu\mu\le\mu)
\end{equation*}
where the $[0,1]$-functor $\mu$ preserves finite weighted colimits. Here, given
$\calR\subseteq [0,1]^X$,
\begin{equation*}
  \mu(\alpha)=
  \bigwedge\{\varphi\mid \varphi\in\calR,\quad \alpha\le\varphi\},
\end{equation*}
and, for a monad $\mu \colon [0,1]^X \longrightarrow [0,1]^X$,
\begin{equation*}
  \calR=\{\alpha\in [0,1]^X\mid\mu(\alpha)=\alpha\}.
\end{equation*}
For a subset $A\subseteq X$, we write $\chi_A \colon X \longrightarrow [0,1]$
for the characteristic function of $A$. The following key result is essentially
\citep[Proposition 1.6.5]{Low97}.

\begin{proposition}\label{d:prop:2}
  Let $\mu,\mu' \colon [0,1]^X \longrightarrow [0,1]^X$ be monads that preserve
  finite weighted colimits. Then $\mu=\mu'$ if and only if
  $\mu(\chi_A)=\mu'(\chi_A)$, for all $A\subseteq X$.
\end{proposition}

Note that, for $\calR\subseteq [0,1]^X$ closed under weighted limits and finite
weighted colimits and with corresponding monad $\mu$, we have
\begin{equation*}
  \mu(\chi_A)(x) =\bigwedge\{\varphi\mid
  \varphi\in\calR,\chi_A\le\varphi\} =\bigwedge\{\varphi\mid
  \varphi\in\calR \text{ and, for all $z\in A$, }\varphi(z)=1\},
\end{equation*}
for all $x\in X$. For a $\thU$-category $(X,a)$, the monad $\mu$ corresponding
to \eqref{d:eq:2} is given by
\begin{equation*}
  \mu(\alpha)(x)=\bigvee_{\fx\in UX}a(\fx,x)\luk \xi U\alpha(\fx),
\end{equation*}
for all $\alpha\in [0,1]^X$. In particular, for every $A\subseteq X$,
\begin{equation*}
  \mu(\chi_A)(x)
  =\bigvee_{\fx\in UX}a(\fx,x)\luk \xi U\chi_A(\fx),\\
  =\bigvee_{\fx\in UA}a(\fx,x),
\end{equation*}
for all $x\in X$.

\begin{lemma}\label{d:lem:3}
  Let $\calR\subseteq [0,1]^X$ be a $[0,1]_\luk$-subcategory closed under
  weighted limits and finite weighted colimits and
  $a \colon UX\times X \longrightarrow [0,1]$ be the initial convergence induced
  by the cone $(\varphi \colon X \longrightarrow [0,1])_{\varphi\in\calR}$ in
  $\Cats{\thU_\luk}$. Then the following assertions hold.
  \begin{enumerate}
  \item
    $a(\fx,x)=\bigwedge\{\varphi(x)\mid \varphi\in\calR,\,\xi
    U\varphi(\fx)=1\}$, for all $\fx\in UX$ and $x\in X$.
  \item For all $A\subseteq X$ and $x\in X$,
    \begin{equation*}
      \bigwedge\{\varphi(x)\mid \varphi\in\calR \text{ and, for all
        $z\in A$, }\varphi(z)=1\} =\bigvee_{\fx\in UA}a(\fx,x).
    \end{equation*}
  \end{enumerate}
\end{lemma}
\begin{proof}
  To see the first statement, note that
  \begin{equation*}
    a(\fx,x) =\bigwedge\{\hom(\xi U\varphi(\fx),\varphi(x))\mid
    \varphi\in\calR\} \le\bigwedge\{\varphi(x)\mid
    \varphi\in\calR,\,\xi U\varphi(\fx)=1\}.
  \end{equation*}
  On the other hand, for every $\varphi\in\calR$, put $u=\xi
  U\varphi(\fx)$. Then $\hom(u,\varphi)\in\calR$ and
  $\xi U(\hom(u,\varphi))(\fx)=1$, which proves the assertion. Regarding the
  second statement, the inequality
  \begin{equation*}
    \bigwedge\{\varphi(x)\mid \varphi\in\calR \text{ and, for all
      $z\in A$, }\varphi(z)=1\} \ge \bigvee_{\fx\in UA}a(\fx,x)
  \end{equation*}
  is certainly true. To see the opposite inequality, put
  \begin{equation*}
    u=\bigwedge\{\varphi(x)\mid \varphi\in\calR \text{ and, for all
      $z\in A$, }\varphi(z)=1\}.
  \end{equation*}
  Let $v<u$ und put $\varepsilon=u-v$, then $\hom(u,v)=1-\varepsilon$. For every
  $\varphi\in\calR$ with $\varphi(x)<v$, there exists some $z\in A$ with
  $\varphi(z)<1-\varepsilon$. In fact, if $\varphi(z)\ge 1-\varepsilon$ for all
  $z\in A$, then $\hom(1-\varepsilon,\varphi(z))=1$ for all $z\in A$, but
  $\hom(1-\varepsilon,\varphi(x))=\varphi(x)+\varepsilon<u$. Therefore
  \begin{equation*}
    \ff=\{\varphi^{-1}([0,1-\varepsilon])\mid \varphi\in\calR,
    \varphi(x)<v\}\cup\{A\}
  \end{equation*}
  is a filter base, let $\fx$ be an ultrafilter finer than $\ff$. Then, for
  every $\varphi\in\calR$ with $\varphi(x)<v$,
  $\xi U\varphi(\fx)\le 1-\varepsilon$. Therefore
  \begin{equation*}
    a(\fx,x)=\bigwedge\{\varphi(x)\mid \varphi\in\calR,\,\xi
    U\varphi(\fx)=1\}\ge v.\qedhere
  \end{equation*}
\end{proof}

From Proposition~\ref{d:prop:2} and Lemma~\ref{d:lem:3} we conclude now:

\begin{corollary}
  Let $\calR\subseteq [0,1]^X$ be a $[0,1]_\luk$-subcategory closed under
  weighted limits and finite weighted colimits. Then
  \begin{equation*}
    \calR=\{\text{all $\thU_\luk$-functors $\varphi \colon X \longrightarrow [0,1]$}\},
  \end{equation*}
  where we consider on $X$ the initial convergence
  $a \colon UX\times X \longrightarrow [0,1]$ induced by $\calR$.
\end{corollary}

\begin{corollary}
  Let $\calR,\calR'\subseteq [0,1]^X$ be $[0,1]_\luk$-subcategories closed under
  weighted limits and finite weighted colimits. If $\calR$ and $\calR'$ induce
  the same convergence, then $\calR=\calR'$.
\end{corollary}

We return now \([0,1]_{\luk}\)-enriched Priestley spaces.

\begin{corollary}
  Let $X$ be in $\Priests{[0,1]_\luk}$ and $\calR$ be the closure of
  $\Priests{[0,1]_\luk}(X,[0,1])$ in $[0,1]^X$ under infima. Then the
  $[0,1]_\luk$-subcategory $\calR\subseteq [0,1]^X$ is closed under weighted
  limits and finite weighted colimits.
\end{corollary}
\begin{proof}
  Since the maps
  \begin{align*}
    \vee \colon [0,1]\times [0,1] & \longrightarrow [0,1],
    & [0,1] &\longrightarrow [0,1],\;u \longmapsto 0,\\
    \wedge \colon [0,1]\times [0,1] & \longrightarrow [0,1],
    & [0,1] &\longrightarrow [0,1],\;u \longmapsto 1
  \end{align*}
  as well as the maps
  \begin{displaymath}
    \hom(u,-)\colon [0,1] \longrightarrow [0,1]
    \quad\text{and}\quad
    u\luk- \colon [0,1] \longrightarrow [0,1]
    \qquad (u\in [0,1])
  \end{displaymath}
  are morphisms in $\Priests{[0,1]_\luk}$, the \([0,1]_{\luk}\)-subcategory
  $\Priests{[0,1]_\luk}(X,[0,1])$ of $[0,1]^X$ is closed under finite weighted
  limits and finite weighted colimits. Clearly, \(\mathcal{R}\subseteq
  [0,1]^{X}\) is closed under all weighted limits. Since
  \begin{displaymath}
    \left ( \bigwedge_{i\in I}\varphi_{i}\right)\vee \left( \bigwedge_{i\in J}\varphi_{j}\right)
    =\bigwedge_{(i,j)\in I\times J}(\varphi_{i}\vee\varphi_{j}),
  \end{displaymath}
  \(\mathcal{R}\) is closed in $[0,1]^X$ under binary suprema, and
  \(\mathcal{R}\) is closed in $[0,1]^X$ under tensors since \(u\luk-\)
  preserves non-empty infima.
\end{proof}

\begin{corollary}
  Let $X$ be in $\Priests{[0,1]_\luk}$. Then every $\thU_\luk$-functor
  $X \longrightarrow [0,1]$ is an infimum of morphisms $X \longrightarrow [0,1]$
  in $\Priests{[0,1]_\luk}$.
\end{corollary}
\begin{proof}
  Since $[0,1]\simeq [0,1]^\op$ in $\Cats{[0,1]_\luk}^\mU$ and $X$ is Priestley,
  the cone $\Priests{[0,1]_\luk}(X,[0,1])$ is point-separating and initial with
  respect to $\CatCHs{[0,1]_\luk} \longrightarrow\COMPHAUS$. Then, since the
  functor $K \colon \CatCHs{[0,1]_\luk}\longrightarrow\Cats{\thU_\luk}$
  preserves initial mono-cones, the closure of $\Priests{[0,1]_\luk}(X,[0,1])$
  in $[0,1]^X$ under infima coincides with $\Cats{\thU_\luk}(X,[0,1])$.
\end{proof}

Using the isomorphism $(-)^\bot\colon [0,1] \longrightarrow [0,1]^\op$, we
obtain the desired result.

\begin{corollary}
  For every $X$ in $\Priests{[0,1]_\luk}$, the inclusion
  $\ftC X\hookrightarrow\Cats{\thU_\luk}(X,[0,1]^\op)$ is
  $\bigvee$-dense. Therefore, for every $\thU_\luk$-functor
  $\varphi \colon X \longrightarrow [0,1]$, the $[0,1]_\luk$-functor
  \begin{displaymath}
    (-\cdot\varphi)\colon\Cats{\thU_\luk}(X,[0,1]^\op)\longrightarrow [0,1]
  \end{displaymath}
  preserves finite weighted limits if and only if the $[0,1]_\luk$-functor
  $(-\cdot\varphi)\colon\ftC X \longrightarrow [0,1]$ does so.
\end{corollary}

We let $\FinLats{[0,1]_{\luk}}$ denote the category of finitely complete and
finitely cocomplete $[0,1]_{\luk}$-categories and $[0,1]_{\luk}$-functors that
preserve finite weighted limits and colimits. We note that
$\FinLats{[0,1]_{\luk}}$ is a $\aleph_1$-ary quasivariety which can be shown as
in \citep[Remark~2.10]{HN18} by adding operations and equations for powers and
finite infima. From the results above we obtain:

\begin{theorem}\label{d:thm:2}
  The fully faithful functor
  \begin{equation*}
    \ftC \colon\left(\Priests{[0,1]_\luk}_\mV\right)^\op
    \longrightarrow\FinSups{[0,1]_\luk}
  \end{equation*}
  restricts to a fully faithful adjoint functor
  \begin{equation*}
    \ftC \colon (\Priests{[0,1]_\luk})^\op\longrightarrow\FinLats{[0,1]_\luk}.
  \end{equation*}
\end{theorem}

\section{Duality theory for enriched Priestley spaces: abstractly}
\label{sec:dual-theory-enrich-abstr}

In Section~\ref{sec:dual-theory-enrich-conc} we presented some duality results
for the category \(\Priests{[0,1]_\luk}\) which in particular expose some
algebraic flavour of \(\Priests{[0,1]_\luk}^{\op}\). For a general quantale
\(\V\), we are still far away from concrete duality results, and in this section
we investigate properties of \(\V\)-categorical compact Hausdorff spaces which
help us to recognise \(\big(\Priests{\V}\big)^{\op}\) as some sort of algebraic
category.

Since we will use it frequently, below we recall an intrinsic characterisation
of cofiltered limits in \(\COMPHAUS\) which goes back to \citep{Bou42}. We refer
to this result commonly as the \emph{Bourbaki-criterion}.

\begin{theorem}
  Let $D \colon I \longrightarrow \COMPHAUS$ be a cofiltered diagram. Then a
  cone $(p_i \colon L \longrightarrow D(i))_{i\in I}$ for $D$ is a limit cone if
  and only if
  \begin{enumerate}
  \item $(p_i \colon L \longrightarrow D(i))_{i\in I}$ is point-separating, and
  \item for every $i\in I$,
    \begin{displaymath}
      \bigcap_{j{\to}i}\im D(j\to i)=\im p_i.
    \end{displaymath}
    That is, ``the image of each $p_i$ is as large as possible''.
  \end{enumerate}
\end{theorem}

\begin{remark}
  The second condition above is automatically satisfied if
  \(p_{i}\colon X \longrightarrow D(i)\) is surjective.
\end{remark}

\begin{remark}
  The Bourbaki-criterion applies also to complete categories \(\catA\) with a
  limit preserving faithful functor
  \(\ftII{-} \colon \catA \longrightarrow \COMPHAUS \). In this case, the first
  condition above reads as
  \begin{quote}
    $(p_i \colon L \longrightarrow D(i))_{i\in I}$ is point-separating and
    initial with respect to the functor
    \(\ftII{-} \colon \catA \longrightarrow \COMPHAUS \).
  \end{quote}
\end{remark}

\begin{example}
  From the Bourbaki-criterion it follows at once that, for instance, every
  Priestley space \(X\) is a cofiltered limit of finite Priestley spaces. In
  fact, let \((p_{i}\colon X \longrightarrow X_{i})_{i\in I}\) be the canonical
  cone for the canonical diagram of \(X\) with respect to all finite
  spaces. Clearly, the cone \((p_{i}\colon X \longrightarrow X_{i})_{i\in I}\)
  is point-separating and initial since \(\two\) is finite. For every index
  \(i\), consider the image factorisation of \(p_{i}\).
  \begin{displaymath}
    \begin{tikzcd}
      & X
      \ar[d,"p_{i}"]\ar[dr,"",->>]\\
      \text{finite spaces:} & X_{i} & \im(p_{i}) \ar[l,"",>->]
    \end{tikzcd}
  \end{displaymath}
  Since \(\im(p_{i})\hookrightarrow X_{i}\) belongs to the diagram, the second
  condition is satisfied.

  We can deduce in a similar fashion the well-known facts that every Boolean
  space $X$ is a cofiltered limit of finite spaces, every compact Hausdorff
  space is a cofiltered limit of metrizable compact Hausdorff spaces, and so on.
\end{example}

\begin{remark}\label{d:rem:4}
  The classic \citeauthor{Sto38}/\citeauthor{Pri70} duality
  \(\PRIEST^{\op}\sim\DLAT\) implies in particular that $\PRIEST^{\op}$ is a
  finitary variety, a fact which can also be seen abstractly using
  Theorems~\ref{d:thm:6} and \ref{d:thm:9}. Below we explain the argument in
  some detail as it serves as a motivation for the investigation in the
  remainder of this section.
  \begin{enumerate}
  \item \(\PRIEST\) has all limits and colimits. This is well-known, but we
    stress that it is a special case of Corollary~\ref{d:cor:1}.
  \item\label{d:item:2} Every embedding in \(\PRIEST\) is a regular
    monomorphism; therefore the class of embeddings coincides with the class of
    regular monomorphisms. We use the argument of \citep[Lemma~4.8]{Hof02}: for
    an embedding \(m \colon X \longrightarrow Y\) in \(\PRIEST\), consider a
    presentation \((q_{i} \colon Y \longrightarrow Y_{i})_{i\in I}\) as a
    cofiltered limit of finite Priestley spaces (= finite partially ordered
    sets). For every \(i\in I\), take the (surjetive, embedding)-factorisation
    \begin{displaymath}
      X\xrightarrow{\quad p_{i}\quad} X_{i}\xrightarrow{\quad m_{i}\quad} Y_{i}
    \end{displaymath}
    of \(q_{i}\cdot m\). Then also
    \((p_{i} \colon X \longrightarrow X_{i})_{i\in I}\) is a limit cone (by the
    Bourbaki-criterion); moreover, \(m\) is the limit of the family
    \((m_{i})_{i\in I}\).
    \begin{equation}\label{d:eq:5}
      \begin{tikzcd}
        X \ar[r,"m",>->]\ar[d,"p_{i}"',->>] & Y
        \ar[d,"q_{i}"] \\
        X_{i} \ar[r,>->,"m_{i}"'] & Y_{i}
      \end{tikzcd}
    \end{equation}
    Having finite and hence discrete domain and codomain, each
    \(m_{i}\colon X_{i} \longrightarrow Y_{i}\) is a regular monomorphism in
    \(\POST_{\fin}=\PRIEST_{\fin}\) (this is a special case of
    Theorem~\ref{d:thm:8}) and therefore also in \(\PRIEST\). Consequently, also
    \(m=\lim_{i}m_{i}\) is a regular monomorphism in \(\PRIEST\).
  \item By definition and by the above, the two-element space is a regular
    cogenerator in \(\PRIEST\).
  \item The two-element space is finitely copresentable in \(\PRIEST\). This is
    very well-known; for our purpose we mention here that it is a consequence of
    \citep[Lemma~2.2]{Hof02a}. In this section we observe that this result
    generalises beyond the finitary case (see Lemma~\ref{d:lem:4}).
  \item\label{d:item:3} The two-element space is regular injective in \(\PRIEST\). This follows
    immediately from finite copresentability: Consider a regular monomorphism
    \(m \colon X \longrightarrow Y\) in \(\PRIEST\) together with
    \eqref{d:eq:5}, and let \(f \colon X \longrightarrow\two\) be a morphism in
    \(\PRIEST\). Since \(\two\) is finitely copresentable, there is some
    \(i_{0}\in I\) and a morphism
    \(\bar{f}\colon X_{i_{0}} \longrightarrow\two\) with
    \(\bar{f}\cdot p_{i_{0}}=f\). Since \(\two\) is injective in \(\POST\) (we
    stress that this is a special case of Proposition~\ref{d:prop:1}), there is
    some \(\bar{g}\colon X_{i_{0}} \longrightarrow Y_{i_{0}}\) with
    \(\bar{g}\cdot m_{i_{0}}=\bar{f}\). Hence, \(\bar{g}\cdot q_{i_{0}}\) is an
    extension of \(f\) along \(m\).
    \begin{displaymath}
      \begin{tikzcd}
        X \ar[r,"m",>->]\ar[d,"p_{i_{0}}"',->>]\ar[ddr,bend right=100,"f"'] & Y
        \ar[d,"q_{i_{0}}"] \\
        X_{i_{0}} \ar[r,>->,"m_{i_{0}}"']\ar[dr,dashed,"\bar{f}"']
        & Y_{i_{0}}\ar[d,dotted,"\bar{g}"]\\
        & \two
      \end{tikzcd}
    \end{displaymath}
  \item \(\PRIEST\) has effective equivalence corelations. A direct proof, even
    for partially ordered compact Hausdorff spaces in general, can be found in
    \citep{AR19_tmp}.
  \end{enumerate}
\end{remark}

Note that our treatment of properties of \(\PRIEST\) rests on results about
\(\ORD\) and \(\POST\), therefore we have first a look at \(\V\)-categories.

\begin{theorem}
  \(\Cats{\V}^{\op}\) is a quasivariety.
\end{theorem}
\begin{proof}
  First recall from Theorem~\ref{d:thm:8} that the regular monomorphisms in
  \(\Cats{\V}\) are precisely the embeddings, and from
  Proposition~\ref{d:prop:1} that \(\V\) is injective and
  \((f \colon X \longrightarrow \V)_{f}\) is initial, for every \(\V\)-category
  \(X\). Moreover, \(\V_{I}\) (indiscrete structure) is a cogenerator in
  \(\Cats{\V}\) and therefore \(\V\times\V_{I}\) is a regular injective regular
  cogenerator. Since \(\Cats{\V}\) is also complete, the assertion follows.
\end{proof}

\begin{remark}
  The observation above should be compared to the fact that ``\(\TOP^{\op}\) is
  a quasivariety'', for details see \citep{BP95, BP96} and \citep{AP97, PW99}.
\end{remark}

On the other hand, the quasivariety \(\Cats{\V}^{\op}\) does not have any
rank. To see this, we recall first the following result from
\citep[Page~64]{GU71} (see also \citep{Ulm71a}).

\begin{proposition}
  A set is copresentable in \(\SET\) if and only if it is a singleton.
\end{proposition}

The corresponding result for \(\Cats{\V}\) is now an immediate consequence of
the following observation.

\begin{proposition}
  \label{d:prop:5}
  The ``discrete'' functor \(\ftD \colon\SET \longrightarrow\Cats{\V}\)
  preserves non-empty limits, in particular cofiltered limits. If \(k=\top\) is
  the top-element of \(\V\), then \(\ftD\) preserves also the terminal object.
\end{proposition}

\begin{corollary}
  If \(X\) is copresentable in \(\Cats{\V}\), then \(\ftII{X}=1\).
\end{corollary}
\begin{proof}
  By Proposition~\ref{d:prop:5}, the forgetful functor
  \(\ftII{-}\colon \Cats{\V}\longrightarrow\SET\) preserves copresentable
  objects since, for every \(\V\)-category \(X\),
  \(\hom(-,\ftII{X})\simeq\hom(D-,X)\).
\end{proof}

We turn now our attention to separated \(\V\)-categories (see \citep{HT10}, for
instance).

\begin{theorem}
  The full subcategory \(\Cats{\V}_{\sep}\) of \(\Cats{\V}\) is closed under
  initial monocones. Therefore the inclusion functor
  \(\Cats{\V}_{\sep} \longrightarrow\Cats{\V}\) has a left adjoint; moreover,
  the canonical forgetful functor \(\Cats{\V}_{\sep} \longrightarrow\SET\) is
  mono-topological with left adjoint
  \(\ftD \colon\SET \longrightarrow\Cats{\V}_{\sep}\) (discrete
  structures). Consequently, \(\Cats{\V}_{\sep}\) is complete and cocomplete,
  with concrete limits. A morphism $f \colon X \longrightarrow Y$ in
  $\Cats{\V}_\sep$ is a monomorphism if and and only if the map \(f\) is
  injective.
\end{theorem}

\begin{remark}
  We do not know if \(\TOP_{0}^{\op}\) or \(\Cats{\V}_{\sep}^{\op}\) are
  quasivarieties. Note that in both cases the class of regular monomorphisms
  does not coincide with the class of embeddings, as we also explain below (see
  also \citep{Bar68}).
\end{remark}

The description of further classes of morphisms in $\Cats{\V}_\sep$ is
facilitated by the notion of \emph{L-closure} introduced in \citep{HT10}.

\begin{lemma}
  \label{d:lem:2}
  Let $X$ be a $\V$-category, $M\subseteq X$ and $x\in X$. Then the following
  assertions are equivalent.
  \begin{tfae}
  \item $x\in\overline{M}$.
  \item For all $f,g \colon X \longrightarrow Y$ in $\Cats{\V}$, if $f|_M=g|_M$,
    then $f(x)\simeq g(x)$.
  \item For all $f,g \colon X \longrightarrow Y$ in $\Cats{\V}$ with \(Y\)
    separated, if $f|_M=g|_M$, then $f(x)=g(x)$.
  \item For all $f,g \colon X \longrightarrow \V$ in $\Cats{\V}$, if
    $f|_M=g|_M$, then $f(x)=g(x)$.
  \end{tfae}
\end{lemma}

\begin{corollary}\label{d:cor:4}
  The epimorphisms in $\Cats{\V}_\sep$ are precisely the L-dense $\V$-functors,
  and the regular monomorphisms the closed embeddings.
\end{corollary}
\begin{proof}
  The assertion regarding epimorphisms is in \citep[Theorem~3.8]{HT10}. However,
  both claims follow immediately from Lemma~\ref{d:lem:2}.
\end{proof}

We denote by \(\Cats{\V}_{\sep,\cc}\) the full subcategory of
\(\Cats{\V}_{\sep}\) formed by all Cauchy complete separated
\(\V\)-categories. The following two results follow immediately from
Corollary~\ref{d:cor:4}.

\begin{corollary}
  A separated \(\V\)-category \(X\) is Cauchy-complete if and only if \(X\) is a
  regular subobject of a power of \(\V\) in \(\Cats{\V}_{\sep}\). Moreover, the
  regular monomorphisms in \(\Cats{\V}_{\sep,\cc}\) are precisely the embeddings
  of \(\V\)-categories.
\end{corollary}

\begin{corollary}
  The \(\V\)-category \(\V\) is a regular injective regular cogenerator in
  \(\Cats{\V}_{\sep,\cc}\). Hence, \(\big(\Cats{\V}_{\sep,\cc}\big)^{\op}\) is a
  quasivariety.
\end{corollary}

\begin{remark}
  Clearly, the ``discrete'' functor
  \(\ftD \colon\SET \longrightarrow\Cats{\V}_{\sep}\) preserves non-empty
  limits.  Under some conditions (see \citep[Proposition~2.2]{CH09}), every
  discrete \(\V\)-category is Cauchy-complete and the discrete functor
  \(\ftD \colon\SET \longrightarrow\Cats{\V}_{\sep,\cc}\) is left adjoint to the
  forgetful functor \(\Cats{\V}_{\sep,\cc}\longrightarrow\SET\) and preserves
  codirected limits. Hence, in this case at most a one-element \(\V\)-category
  can be copresentable in \(\Cats{\V}_{\sep,\cc}\).
\end{remark}

\begin{remark}
  In general, the category \((\Cats{\V}_{\sep,\cc})^{\op}\) is not a variety,
  \emph{i.e.} does not have effective equivalence corelations. A counterexample
  is already given by the case \(\V=\two\) since the dual of
  \(\POST\sim\Cats{\two}_{\sep,\cc}\) is not a variety. This fact is well-known
  and follows immediately from the following facts:
  \begin{itemize}
  \item \(\POST^{\op}\) is equivalent to the category \(\TAL\) of totally
    algebraic lattices and maps preserving all suprema and all infima (see
    \citep{RW94}, for instance),
  \item \(\TAL\) is a full subcategory of the category \(\CCD\) of
    (constructively) completely distributive lattices and maps preserving all
    suprema and all infima,
  \item the unit interval \([0,1]\) is completely distributive but not totally
    algebraic,
  \item the category \(\CCD\) is monadic over \(\SET\) (see \citep{PW99}, and
    \citep{PZ15} for a generalisation to quantaloid-enriched categories). Here
    the free algebra over a set \(X\) is given by the complete lattice of upsets
    of the powerset of \(X\), and this lattice is totally algebraic and
    therefore also the free totally algebraic lattice over \(X\).
  \end{itemize}
\end{remark}

Another important property of \(\V\)-categories and \(\V\)-functors is
established in \citep{KL01}: $\Cats{\V}$ is locally presentable, for every
quantale $\V$. Under Assumption~\ref{d:ass:2} below, and based on
\citep{Sea05,Sea09}, we show that \(\Cats{\V}\) is locally
\(\aleph_{1}\)-copresentable by describing a corresponding countable limit
sketch. This will help us later to identify $\CatCHs{\V}$ as the model category
of a $\aleph_{1}$-ary limit sketch in $\COMPHAUS$. To do so, \emph{in the
  remainder of this section we impose the following conditions on the quantale
  \(\V\)}.

\begin{assumption}\label{d:ass:2}
  We assume that the underlying lattice of \(\V\) is completely distributive,
  and that there is a countable subset \(D\subseteq \V\) so that, for all
  \(v\in\V\),
  \begin{equation*}
    v=\bigvee\{u\in D\mid u\lll v\}.
  \end{equation*}
\end{assumption}

\begin{examples}
  The quantales of Examples~\ref{d:exs:1} and Example~\ref{d:ex:2} satisfy
  Assumption~\ref{d:ass:2}.
\end{examples}

\begin{remark}\label{d:rem:3}
  Under Assumption~\ref{d:ass:2}, for each \(v\in\V\),
  \begin{displaymath}
    \upc v=\bigcap\{\upc u\mid u\in D, u\lll v\}.
  \end{displaymath}
  Hence, by Remark~\ref{d:rem:2}~(\ref{d:item:1}), the sets \(\upc u\) (\(u\in
  D\)) form a subbasis for the closed sets of the dual of the Scott topology of
  \(\V\).
\end{remark}

We start with the following well-known fact.
\begin{lemma}\label{d:lem:1}
  The assignments
  \begin{equation*}
    (\varphi \colon X\to \V) \quad \longmapsto \quad (\varphi^{-1}(\upc u)_{u\in D})
  \end{equation*}
  and
  \begin{equation*}
    (B_u)_{u\in D} \quad \longmapsto \quad
    (\varphi \colon X\to \V,\,
    x \mapsto \bigvee\{u\in D\mid x\in B_u\})
  \end{equation*}
  define a bijection between the sets
  \begin{equation*}
    \V^X\qquad\text{and}\qquad
    \{(B_u)_{u\in D}\mid
    \text{for all $u\in D$, }B_u\subseteq X\; \&\; B_{u}=\bigcap_{v\lll u}B_{v}\}.
  \end{equation*}
\end{lemma}

\begin{remark}
  Under the bijection above, a map $a \colon X\times X \longrightarrow \V$
  corresponds to a family $(R_u)_{u\in D}$ of binary relations $R_u$ on $X$.
\end{remark}

\begin{proposition}
  A \(\V\)-relation $a \colon X\times X \longrightarrow \V$ is reflexive if and
  only if $\Delta_X\subseteq R_k$. Moreover,
  $a \colon X\times X \longrightarrow \V$ is transitive if and only if, for all
  $u,v\in D$, $R_u\cdot R_v\subseteq R_{u\otimes v}$.
\end{proposition}
\begin{proof}
  See \citep{Sea09}.
\end{proof}

\begin{remark}
  A $\V$-category $(X,a)$ is separated if and only if the relation $R_k$ on
  \(X\) is anti-symmetric.
\end{remark}

Therefore the structure of a \(\V\)-category can be equivalently described by a
family of binary relations, suitably interconnected. Since a map
\(f \colon X \longrightarrow Y\) between \(\V\)-categories is a \(\V\)-functor
if and only if \(f\) preserves the corresponding relations, we obtain at once:

\begin{corollary}\label{d:cor:2}
  The categories $\Cats{\V}$ and $\Cats{\V}_{\sep}$ are model categories in
  $\SET$ of an $\aleph_1$-ary countable limit sketch.
\end{corollary}

\begin{remark}
  We do not know yet wether \(\Cats{\V}_{\sep,\cc}\) is locally
  presentable. However, we note that in \citep{AMMU15} this property is proven
  for \(\V=[0,1]_{\luk}\), that is, for the case of bounded metric spaces.
\end{remark}

We turn now our attention to \(\V\)-categorical compact Hausdorff spaces. First
we observe that Proposition~\ref{d:prop:5} as well as some of its consequences
generalise directly to the topological case.

\begin{proposition}
  The ``discrete'' functors \(\ftD \colon\COMPHAUS \longrightarrow\CatCHs{\V}\)
  and \(\ftD \colon\COMPHAUS \longrightarrow\CatCHs{\V}_{\sep}\) preserve
  non-empty limits. If \(k=\top\) is the top-element of \(\V\), then \(\ftD\)
  preserves also the terminal object.
\end{proposition}

Regarding copresentable compact Hausdorff spaces, we recall the following result
from \citep[6.5(c)]{GU71} (see also \citep{Ulm71a}).
\begin{theorem}
  \label{d:thm:5}
  \begin{enumerate}
  \item The finitely copresentable compact Hausdorff spaces are precisely the
    finite ones.
  \item The \(\aleph_1\)-copresentable compact Hausdorff spaces are precisely
    the metrisable ones. In particular, the unit interval \([0,1]\) is
    \(\aleph_{1}\)-copresentable in \(\COMPHAUS\).
  \end{enumerate}
\end{theorem}

\begin{corollary}
  \label{d:cor:5}
  For every regular cardinal \(\lambda\), the forgetful functors
  \(\ftII{-} \colon\CatCHs{\V} \longrightarrow\COMPHAUS\) and
  \(\ftII{-} \colon\CatCHs{\V}_{\sep} \longrightarrow\COMPHAUS\) preserve
  \(\lambda\)-copresentable objects. In particular, every finitely copresentable
  (separated) \(\V\)-categorical compact Hausdorff space is finite and every
  \(\aleph_{1}\)-copresentable (separated) \(\V\)-categorical compact Hausdorff
  space has a metrizable topology.
\end{corollary}

We are particularly interested in properties of the space \(\V\). We start with
the following observation.

\begin{proposition}\label{d:prop:4}
  A subbase for the Lawson topology on \(\V\) is given by the sets
  \begin{equation*}
    \{u\in\V\mid v\lll u\}
    \quad\text{and}\quad
    \{u\in\V\mid v\nleq u\}\qquad (v\in D).
  \end{equation*}
\end{proposition}
\begin{proof}
  By definition, the Lawson topology is the join of the Scott topology and the
  lower topology of \(\V\) (see Remark~\ref{d:rem:2}); we recall that the latter
  is generated by the sets \((\upc v)^{\complement}\), with \(v\in\V\). Since
  the lattice \(\V\) is completely distributive, the Scott topology of \(\V\)
  has as subbase the sets (see Proposition~\ref{p:0})
  \begin{equation*}
    \{u\in\V\mid v\lll u\},
  \end{equation*}
  with \(v\in\V\). Since ``generated topology'' defines a left adjoint, the sets
  \begin{equation*}
    \{u\in\V\mid v\lll u\}
    \quad\text{and}\quad
    \{u\in\V\mid v\nleq u\}\qquad (v\in \V)
  \end{equation*}
  form a subbase for the Lawson topology of \(\V\). Let now \(v\in\V\). For each
  \(v\lll u\in\V\), there is some \(w\in D\) with \(v\lll w\lll u\), therefore
  \begin{equation*}
    \{u\in\V\mid v\lll u\}=\bigcup_{w\in D,v\lll w}\{u\in\V\mid w\lll u\}.
  \end{equation*}
  Finally, since \(v\in\bigvee\{w\in D\mid w\lll v\}\), we obtain
  \(\upc v=\bigcap \{\upc w\mid w\in D, w\lll v\}\) and therefore
  \((\upc v)^{\complement}=\bigcup \{(\upc w)^{\complement}\mid w\in D, w\lll
  v\}\).
\end{proof}

\begin{corollary}
  The Lawson topology makes \(\V\) a $\aleph_{1}$-copresentable object in
  $\COMPHAUS$.
\end{corollary}
\begin{proof}
  By Proposition~\ref{d:prop:4}, the Lawson topology on \(\V\) has a countable
  subbase and therefore also a countable base. Hence, \(\V\) with the Lawson
  topology is a metrizable compact Hausdorff space and therefore, by
  Theorem~\ref{d:thm:5}, $\aleph_{1}$-copresentable in $\COMPHAUS$.
\end{proof}

We shall now extend Corollary~\ref{d:cor:2} to the topological context and show
that \(\CatCHs{\V}\) is a model category of a limit sketch in \(\COMPHAUS\). To
prepare this, we recall an alternative way of expressing the compatibility
between topology and \(\V\)-categories which is closer to \citeauthor{Nac65}'s
original definition.

\begin{proposition}\label{d:prop:3}
  For a $\V$-category $(X,a)$ and a $\mU$-algebra $(X,\alpha)$ with the same
  underlying set $X$, the following assertions are equivalent.
  \begin{tfae}
  \item $\alpha \colon \ftU(X,a) \longrightarrow (X,a)$ is a $\V$-functor.
  \item $a \colon (X,\alpha)\times(X,\alpha) \longrightarrow (\V,\xi_\le)$ is
    continuous.
  \end{tfae}
\end{proposition}
\begin{proof}
  See \citep[Proposition~3.22]{HR18}.
\end{proof}

\begin{lemma}\label{d:lem:6}
  Consider \(\V\) with the dual of the Scott topology. Then, under the
  correspondence of Lemma~\ref{d:lem:1}, $\varphi \colon X \longrightarrow \V$
  is continuous if and only if, for each \(u\in D\), $B_u$ is closed in $X$.
\end{lemma}
\begin{proof}
  Recall from Remark~\ref{d:rem:3} that the sets $\upc u$ ($u\in D$) form a
  subbase for the closed sets of the dual of the Scott topology of $\V$.
\end{proof}

Applying Lemma~\ref{d:lem:6} to the map
$a \colon (X,\alpha)\times(X,\alpha) \longrightarrow (\V,\xi_\le)$ of
Proposition~\ref{d:prop:3} gives immediately:

\begin{theorem}
  Both $\CatCHs{\V}$ and $\CatCHs{\V}_\sep$ are model categories in $\COMPHAUS$
  of a countable $\aleph_1$-ary limit sketch. Hence, both categories are locally
  copresentable.
\end{theorem}
\begin{proof}
  For the second affirmation, use \citep[Remark~2.63]{AR94}.
\end{proof}

\begin{remark}
  At this moment we do not have any information about the rank of the locally
  presentable category $\CatCHs{\V}^{\op}$; in particular, we do not know if
  $\CatCHs{\V}^{\op}$ is \(\aleph_{1}\)-ary locally copresentable.
\end{remark}

In order to obtain more information on copresentable objects in $\CatCHs{\V}$,
we adapt now \citep[Lemma~2.2]{Hof02a} to the case of a general regular
cardinal. Here we call a \(\lambda\)-ary limit sketch
$\calS=(\catC,\calL,\sigma)$ \df{\(\lambda\)-small} whenever there is a set
\(M\) of morphisms in \(\catC\) of cardinality less than \(\lambda\) so that
every morphism of \(\catC\) is a finite composite of morphisms in \(M\). Hence,
for \(\lambda>\aleph_{0}\), we require the category \(\catC\) to be
\(\lambda\)-small.

\begin{lemma}\label{d:lem:4}
  Let $\lambda$ be a regular cardinal and let $\calS=(\catC,\calL,\sigma)$ be a
  $\lambda$-small limit sketch. Then a model of $\calS$ in a category \(\catX\)
  is $\lambda$-copresentable in $\Mod(\calS,\catX)$ provided that each component
  is $\lambda$-copresentable in $\catX$.
\end{lemma}
\begin{proof}
  See \citep[Lemma~2.2]{Hof02a}.
\end{proof}

By Assumption~\ref{d:ass:2}, the limit sketch for $\CatCHs{\V}$ is countable
which allows us to derive the following properties.

\begin{corollary}
  A \(\V\)-categorical compact Hausdorff space is $\aleph_1$-ary copresentable
  in $\CatCHs{\V}$ (respectively $\CatCHs{\V}_{\sep}$) if and only if its
  underlying compact Hausdorff space is metrizable. In particular, $\V^{\op}$ is
  $\aleph_1$-ary copresentable in $\CatCHs{\V}$ and in $\CatCHs{\V}_\sep$.
\end{corollary}

\begin{corollary}
  \label{d:cor:3}
  If the quantale \(\V\) is finite, then the finitely copresentable objects of
  $\CatCHs{\V}$ (respectively $\CatCHs{\V}_{\sep}$) are precisely the finite
  ones.
\end{corollary}

\begin{remark}
  The conclusion of Lemma~\ref{d:lem:4} is not necessarily optimal. For
  instance, the circle line $\TT=\quot{\RR}{\ZZ}$ is $\aleph_1$-copresentable
  but not finitely copresentable in $\COMPHAUS$ (see \citep[6.5]{GU71}); hence,
  Lemma~\ref{d:lem:4} implies that $\TT$ is $\aleph_1$-copresentable in the
  category $\COMPHAUSAB$ of compact Hausdorff Abelian groups and continuous
  homomorphisms. However, by the famous Pontryagin duality theorem (see
  \citep{Mor77}, for instance), $\TT$ is even finitely copresentable in
  $\COMPHAUSAB$ which cannot be concluded from Lemma~\ref{d:lem:4}.
\end{remark}

\begin{remark}\label{d:rem:1}
  In particular, the finitely copresentable partially ordered compact spaces are
  precisely the finite ones. Moreover, a partially ordered compact space is
  \(\aleph_{1}\)-copresentable in \(\POSCH\) if and only if its underlying
  compact Hausdorff topology is metrisable. This characterisation is slightly
  different from our result in \citep{HNN18} where the
  \(\aleph_{1}\)-copresentable objects in \(\POSCH\) are characterised as those
  spaces where both -- the order and the topology -- are induced by the same
  (not necessarily symmetric) metric.
\end{remark}

The results above also imply that the reflector
$\pi_0 \colon \CatCHs{\V} \longrightarrow\Priests{\V}$ preserves
$\aleph_1$-cofiltered limits. In the classical case, the corresponding property
is shown in \citep[Page~67]{GU71} using Stone duality; however, our proof here
is based on the Bourbaki-criterion.

\begin{proposition}
  \label{d:prop:6}
  The reflection functor $\pi_0 \colon \CatCHs{\V} \longrightarrow\Priests{\V}$
  preserves $\aleph_1$-cofiltered limits (and even cofiltered limits if \(\V\)
  is finite).
\end{proposition}
\begin{proof}
  Let $(p_i \colon X \longrightarrow D(i))_{i\in I}$ be a $\aleph_1$-cofiltered
  limit in $\CatCHs{\V}$ (\(\aleph_{0}\)-cofiltered if \(\V\) is finite). Since
  $\V^{\op}$ is $\aleph_1$-ary copresentable ($\aleph_0$-ary copresentable if
  \(\V\) is finite) in $\CatCHs{\V}$, the cone of all morphisms of type
  $X \longrightarrow \V^{\op}$ is given by the cone of all morphism
  \begin{equation*}
    X \xrightarrow{\quad p_i\quad}D(i)\xrightarrow{\quad \varphi\quad} \V^{\op}
  \end{equation*}
  where $i\in I$ and $\varphi \colon D(i) \longrightarrow \V^{\op}$ in
  $\CatCHs{[0,1]}$. Hence, for every $i\in I$ and every
  $\varphi \colon X \longrightarrow \V^{\op}$, we obtain the commutative diagram
  \begin{equation*}
    \begin{tikzcd}[ampersand replacement=\&]
      X \ar[r,"",->>]\ar[d,"p_i"'] \& \pi_0(X)
      \ar[d,"\pi_0(p_i)"] \ar[dr,bend left]\\
      D(i) \ar[r,""',->>] \& \pi_0(D(i))\ar[r,"\widetilde{\varphi}"'] \&
      {\V^{\op}}.
    \end{tikzcd}
  \end{equation*}
  Therefore the cone
  $(\pi_0(p_i) \colon \pi_0(X) \longrightarrow \pi_0(D(i)))_{i\in I}$ is initial
  with respect to the forgetful functor $\CatCHs{\V} \longrightarrow\COMPHAUS$.

  Let now $i\in I$ and $x\in D(i)$ with
  $x\in\bigcap\{\im(\pi_0(D(k)))\mid k \colon j\to i\text{ in }I\}$. Let
  $A\subseteq X$ be the inverse image of $x$ under the reflection map
  $D(i) \longrightarrow \pi_0(D(i))$. Then, for every $k \colon j\to i$ in $I$,
  $\varnothing\neq A\cap\im(k)$. Since the set $\{\im(k)\mid k \colon j\to i\}$
  is codirected and $A$ is compact, we obtain
  \begin{equation*}
    \varnothing\neq\bigcap_{k \colon j\to i}A\cap\im D(k)
    =A\cap\bigcap_{k \colon j\to i}\im D(k)=A\cap\im(p_i).
  \end{equation*}
  Therefore $x\in\im(\pi_{0}(p_{i}))$.
\end{proof}

Combining Corollaries~\ref{d:prop:6} and \ref{d:cor:5} we obtain:

\begin{corollary}
  \label{d:cor:6}
  \begin{enumerate}
  \item An object is $\aleph_1$-ary copresentable in $\Priests{\V}$ if and only
    if its underlying compact Hausdorff space is metrizable. In particular,
    $\V^{\op}$ is $\aleph_1$-ary copresentable in $\Priests{\V}$.
  \item Assume that \(\V\) is finite. Then an object is finitely copresentable
    in $\Priests{\V}$ if and only if it is finite. In particular, $\V^{\op}$ is
    finitely copresentable in $\Priests{\V}$.
  \end{enumerate}
\end{corollary}
\begin{proof}
  Since the left adjoint $\pi_0 \colon \CatCHs{\V} \longrightarrow\Priests{\V}$
  of $\Priests{\V}\longrightarrow\CatCHs{\V}$ preserves
  \(\aleph_{1}\)-codirected limits, the inclusion functor
  $\Priests{\V}\longrightarrow\CatCHs{\V}$ preserves
  \(\aleph_{1}\)-copresentable objects. Furthermore, since \(\Priests{\V}\) is
  closed in \(\CatCHs{\V}\) under limits,
  $\Priests{\V}\longrightarrow\CatCHs{\V}$ reflects \(\aleph_{1}\)-copresentable
  objects. The second affirmation follows similarly.
\end{proof}

\begin{theorem}
  \label{d:thm:4}
  The category $\Priests{\V}$ is locally $\aleph_1$-ary copresentable. If \(\V\)
  is finite, then $\Priests{\V}$ is even locally $\aleph_0$-ary copresentable.
\end{theorem}
\begin{proof}
  By the Bourbaki-criterion, every $X$ in $\Priests{\V}$ is a limit of the
  canonical diagram of $X$ with respect to the full subcategory of
  \(\Priests{\V}\) defined by all \(\aleph_{1}\)-copresentable objects. Since
  $\Priests{\V}$ is complete, we conclude that $\Priests{\V}$ is locally
  $\aleph_1$-ary copresentable. If \(\V\) is finite, the same argument works
  with \(\aleph_{0}\) \emph{in lieu} of \(\aleph_{1}\).
\end{proof}

\begin{remark}
  By Corollary~\ref{d:cor:6}, the fully faithful functor
  \begin{equation*}
    \ftC \colon \Priests{[0,1]_\luk}\longrightarrow\FinLats{[0,1]_\luk}^\op
  \end{equation*}
  of Theorem~\ref{d:thm:2} preserves $\aleph_1$-filtered limits which allows for
  an alternative proof of Corollary~\ref{d:prop:6} for $\otimes=\luk$: Firstly,
  the dualising object \([0,1]\) induces a natural dual adjunction (see
  \citep{PT91})
  \begin{displaymath}
    \begin{tikzcd}[column sep=huge] 
      \CatCHs{[0,1]_{\luk}} %
      \ar[shift left, start anchor=east, end anchor=west, bend left=25,
      ""{name=U,below}]%
      {r}{\ftC=\hom(-,[0,1])} %
      & \FinLats{[0,1]_\luk}^\op %
      \ar[start anchor=west,end anchor=east,shift left,bend left=25,
      ""{name=D,above}] %
      {l}{\hom(-,[0,1])} \ar[from=U,to=D,"\bot" description]
    \end{tikzcd}
  \end{displaymath}
  where the fixed subcategory on the left-hand side is precisely
  \(\Priests{[0,1]_\luk}\). Then the functor
  $\pi_0 \colon\CatCHs{[0,1]_\luk} \longrightarrow\Priests{[0,1]_\luk}$ is the
  composite of the functor
  $\ftC \colon \CatCHs{[0,1]_\luk} \longrightarrow\FinLats{[0,1]_\luk}^\op$ and
  the right adjoint functor
  $\FinLats{[0,1]_\luk}^\op \longrightarrow\Priests{[0,1]_\luk}$ above (see
  \citep[Theorem~2.0]{LR79}, and note that, for every \(L\) in
  \(\FinLats{[0,1]_\luk}\), the space \(\hom(L,[0,1])\) is Priestley by
  construction).
\end{remark}

Next, we link \(\V\)-categorical compact Hausdorff spaces with compact
\(\V\)-categories. To do so, we also \emph{impose now the following condition}.
\begin{assumption}
  \label{d:ass:3}
  For the neutral element \(k\) of \(\V\), the set
  \begin{equation*}
    \{u\in\V\mid u\lll k\}
  \end{equation*}
  is directed.
\end{assumption}

Then \(\bot<k\) and, for all \(u,v\in\V\),
\begin{equation*}
  k\le u\vee v \implies (k\le u\quad\text{ou}\quad k\le v);
\end{equation*}
which guarantees that the L-closure is topological (see
\citep[Proposition~3.3]{HT10}). Moreover, under this condition, a separated
\(\V\)-category $X$ induces a Hausdorff topology; if this topology is compact,
\(X\) becomes a \(\V\)-categorical compact Hausdorff space (see
\citep[][Theorem~3.28 and Propositions~3.26 and 3.29]{HR18}). We let
\(\Cats{\V}_{\sep,\comp}\) denote the full subcategory of \(\Cats{\V}_{\sep}\)
defined by those \(\V\)-categories with compact topology, then this construction
defines a fully faithful functor
\begin{equation*}
  \Cats{\V}_{\sep,\comp} \longrightarrow \CatCHs{\V}_{\sep}.
\end{equation*}

From Lemma~\ref{d:lem:2} and Corollary~\ref{d:cor:4} we obtain immediately:

\begin{corollary}
  Let $f \colon X \longrightarrow Y$ be in $\Cats{\V}_{\sep,\comp}$. Then
  \begin{enumerate}
  \item \(f\) is a regular monomorphism in $\CatCHs{\V}_\sep$ if and only if
    \(f\) is an embedding, and
  \item \(f\) is an epimorphism in $\CatCHs{\V}_\sep$ if and only if \(f\) is
    surjective.
  \end{enumerate}
\end{corollary}

\begin{lemma}
  If the \(\V\)-category \(\V\) is compact, then the L-topology on \(\V\)
  coincides with the Lawson topology.
\end{lemma}
\begin{proof}
  By \citep[][Remark~4.27]{HN20}, for every \(u \in \V\), the sets \(\upc u\)
  and \(\downc u\) are closed in \(\V\) with respect to the L-closure.
\end{proof}

\begin{example}
  In particular, the L-closure on the $[0,1]_{\luk}$-category $[0,1]$ induces
  the Euclidean topology with convergence $\xi$.
\end{example}

\begin{corollary}
  Assume that the \(\V\)-category \(\V\) is compact. Then we have a fully
  faithful functor
  \begin{displaymath}
    \Cats{\V}_{\sep,\comp} \longrightarrow\Priests{\V},
  \end{displaymath}
  and every \(\V\)-enriched Priestley space is a cofiltered limit of compact
  separated \(\V\)-categories. Moreover:
  \begin{itemize}
  \item every embedding $f \colon X \longrightarrow Y$ in $\Priests{\V}$ is a
    regular monomorphism, and
  \item therefore the epimorphisms in $\Priests{\V}$ are precisely the
    surjective morphisms.
  \end{itemize}
  Consequently, \(\V^{\op}\) is a regular cogenerator in \(\Priests{\V}\).
\end{corollary}
\begin{proof}
  Regarding embeddings, we use the same argument as in
  Remark~\ref{d:rem:4}~(\ref{d:item:2}). Every epimorphism \(e\) in
  \(\Priests{\V}\) factorises as \(e=m\cdot g\) where \(g\) is surjective and
  \(m\) is a regular monomorphism, hence \(m\) is an isomorphism and therefore
  \(e\) is surjective.
\end{proof}

\begin{remark}
  If \(\V\) is finite, then \(\V^{\op}\) is finitely copresentable in
  \(\Priests{\V}\) and, with the same argument as in
  Remark~\ref{d:rem:4}~(\ref{d:item:3}), we deduce that \(\V^{\op}\) is regular
  injective in \(\Priests{\V}\). Unfortunately, the same argument does not seem
  to work if \(\V\) is infinite since in this case
  \begin{itemize}
  \item \(\V^{\op}\) is countably but in general not finitely copresentable in
    \(\Priests{\V}\), but
  \item we are not able to prove that every \(\V\)-enriched Priestley space is a
    \(\aleph_{1}\)-cofiltered limit of compact separated \(\V\)-categories.
  \end{itemize}
\end{remark}

We finish this paper by bringing another well-known result from order theory
into the enriched realm: every \(\V\)-categorical compact Hausdorff space is a
quotient of a Priestley space. We shall make use of the free \(\V\)-categorical
compact Hausdorff space, for \(\thU\)-category \((X,a)\), and therefore assume
that our \emph{topological theory \(\thU=\utheory\) is strict} in the sense of
\citep{Hof07}:

\begin{assumption}
  The complete lattice \(\V\) is completely distributive, and we consider the
  Lawson topology \(\xi \colon U\V \longrightarrow \V\) (see
  Remark~\ref{d:rem:2}). Furthermore, the tensor
  \(\otimes \colon \V\times\V\to\V\) is continuous with respect to the Lawson
  topology.
\end{assumption}

We consider the free \(\V\)-categorical compact Hausdorff space
\begin{displaymath}
  (\ftU X,\widehat{a},m_{X})
\end{displaymath}
of a \(\thU\)-category \((X,a)\) where \(\widehat{a}=\ftU a\cdot m_{X}^{\circ}\)
(see \citep[Theorem~III.5.3.5]{HST14}). Moreover, by \citep[Lemma~6.7 and
Proposition~6.11]{Hof07}, the map
\begin{displaymath}
  \delta_{A}\colon X \longrightarrow\V,\quad
  x \longmapsto \bigvee\{a(\fx,x)\mid \fx\in \ftU A\}
\end{displaymath}
is an \(\thU\)-functor, for every \(A\subseteq X\), since it can be written as
the composite
\begin{displaymath}
  \begin{tikzcd}[column sep=large]
    X\ar{r}{\mate{a}} \ar[bend left=40]{rr}{\delta_{A}}
    & \V^{\ftU A}\ar{r}{\bigvee} & \V.
  \end{tikzcd}
\end{displaymath}
For our next result we need to consider a stronger version of
Assumption~\ref{d:ass:3} which \emph{we assume from now on}:

\begin{assumption}
  The set \(\{u\in\V\mid u\lll v\}\) is directed, for every \(v\in\V\).
\end{assumption}

\begin{lemma}\label{d:lem:5}
  For every \(\thU\)-category \((X,a)\) and all \(\fx,\fy\in \ftU X\),
  \begin{displaymath}
    \widehat{a}(\fx,\fy)=\bigvee\{u\in \V
    \mid \forall A\in\fx\,.\,\delta_{A}^{-1}(\upc u)\in \fy\}.
  \end{displaymath}
\end{lemma}
\begin{proof}
  Same as in \citep[page~83]{Hof13}, which in turn relies on
  \citep[Corollary~1.5]{Hof06}.
\end{proof}

\begin{lemma}
  For every \(\thU\)-category \((X,a)\), the cone
  \begin{displaymath}
    (\ftU X\xrightarrow{\quad\xi\cdot\ftU\delta_{A}\quad}\V)_{A\subseteq X}
  \end{displaymath}
  is initial in \(\CatCHs{\V}\).
\end{lemma}
\begin{proof}
  For all \(\fx,\fy\in\ftU X\), we show that
  \begin{displaymath}
    \widehat{a}(\fx,\fy)\ge \bigwedge\{\xi\cdot\ftU\delta_{A}(\fy)\mid A\in\fx\},
  \end{displaymath}
  and observe that \(\delta_{A}(\fx)\ge k\), for every \(A\in\fx\). Let
  \begin{displaymath}
    u\lll \bigwedge\{\xi\cdot\ftU\delta_{A}(\fy)\mid A\in\fx\}.
  \end{displaymath}
  Then, for every \(A\in\fx\), \(u\lll \xi\cdot\delta_{A}(\fy)\), and therefore
  \(\upc u\in\ftU\delta_{A}(\fy)\), which is equivalent to
  \(\delta_{A}^{-1}(\upc u)\in \fy\). Therefore \(u\le \widehat{a}(\fx,\fy)\),
  by Lemma~\ref{d:lem:5}.
\end{proof}

\begin{corollary}
  For every \(\thU\)-category \((X,a)\), the \(\V\)-categorical compact
  Hausdorff space \((\ftU X)^{\op}\) is Priestley.
\end{corollary}

\begin{corollary}
  Every \(\V\)-categorical compact Hausdorff space is a regular quotient of a
  Priestley space.
\end{corollary}
\begin{proof}
  With \(\alpha \colon\ftU X \longrightarrow X\) denoting the convergence of
  \(X\) (and \(X^{\op}\)),
  \begin{displaymath}
    \alpha \colon \ftU(X^{\op}) \longrightarrow X^{\op}
  \end{displaymath}
  is a regular quotient in \(\CatCHs{\V}\), and hence also
  \(\alpha \colon \ftU(X^{\op})^{\op} \longrightarrow X\).
\end{proof}



\end{document}